\pdfoutput=1
\documentclass{amsart}

\usepackage{amsmath}
\usepackage{amsfonts, latexsym}
\usepackage{amssymb}
\usepackage[all]{xypic}
\usepackage{psfrag}
\usepackage[T1]{fontenc}
\usepackage[numbers]{natbib}
 \usepackage{url}
\usepackage[colorlinks]{hyperref}

\newcommand{\dd}{\mathrm{d}}
\newcommand{\x}{{{x}}}
\renewcommand{\P}{\mathcal{P}}
\newcommand{\C}{\mathcal{C}}
\renewcommand{\L}{\mathcal{L}}
\newcommand{\Q}{\mathcal{Q}}
\newcommand{\R}{\mathbb{R}}
\newcommand{\mmu}{\overline{\mu}}
\newcommand{\nnu}{\overline{\nu}}
\newcommand{\virg}[1]{``#1''}
\newcommand{\Span}[1]{\left\langle#1\right\rangle}
\newcommand{\Th}{^\textrm{th}}
\newcommand{\St}{^\textrm{st}}

\newcommand{\rank}{\textrm{\normalfont rank}\,}

\def\interno{\vbox{\hbox{\vbox to .3 truecm{\vfill\hbox to .25 truecm
{\hfill\hfill}\vfill}\vrule}\hrule}\hskip 2pt}
\newcommand{\dens}[1]{  \partial_{#1}   {\textrm{\scalebox{1.5}{$\lrcorner$}}} {\mathrm d}^n\x }

\newtheorem{theorem}{Theorem}
\newtheorem{proposition}{Proposition}
\newtheorem{corollary}{Corollary}
\newtheorem{definition}{Definition}
\theoremstyle{remark}
\newtheorem{remark}{Remark}
\newtheorem{example}{Example}

\let\phi\varphi

\let\cal\mathcal

\begin{document}

\markboth{J.~Kijowski and G.~Moreno}
{Symplectic structures related with higher order variational problems}

%
%

\title{Symplectic structures related with higher order variational problems}

\author{Jerzy Kijowski}\thanks{The first author was partially supported by Narodowe Centrum Nauki, Poland (NCN Grant DEC-2011/03/B/ST1/02625).}

\address{Center for Theoretical Physics,\\
Polish Academy of Sciences,\\
Al. Lotnik\'ow 32/46,\\
02--668 Warsaw, Poland.\\
\texttt{kijowski@cft.edu.pl}
}

\author{Giovanni Moreno}\thanks{The second author is  thankful to the   Grant Agency of the Czech Republic (GA \v CR) for financial support under the project P201/12/G028.}

\address{Mathematical Institute in Opava,\\
Silesian University in Opava,\\
Na Rybn\'{\i}\v{c}ku 626/1, 746 01 Opava, Czech Republic.\\
\texttt{Giovanni.Moreno@math.slu.cz}
}

\maketitle


\begin{abstract}
In this paper we derive the symplectic framework for field theories defined by higher--order Lagrangians. The construction is based on the symplectic reduction of  suitable spaces of iterated jets. 
 
The possibility of reducing a higher--order system of PDEs to a constrained
first--order one, the  symplectic structures naturally arising in the
dynamics of a first--order Lagrangian  theory, and the importance  of the
Poincar\'e--Cartan form for variational problems, are all well--established
facts. However, their adequate combination corresponding to higher--order theories is missing in the literature. Here we obtain   a   consistent and truly finite--dimensional canonical formalism, as well as a higher--order version of the
Poincar\'e--Cartan form.  In our exposition, the rigorous
global proofs of the main results are always accompanied  by their local coordinate
descriptions,    indispensable   to work out practical
examples. 
\end{abstract}


\tableofcontents

\section{Introduction}

Symplectic geometry was born in Classical Mechanics as a framework to describe its canonical structure, i.e.,~ Poisson brackets, canonical transformations, Hamilton--Jacobi theory, etc. It provides also the basic notions of Quantum Mechanics, according to W.~Heisenberg. Its geometric version, known as {\em Geometric Quantization Theory} (see, e.g.,  \cite{Sou}) is an important tool in the theory of group representations.\par
But also fields (e.g.,~electromagnetic or gravitational) have to be quantized. For this purpose people usually replace the finite--dimensional symplectic space of Cauchy data in Mechanics by its infinite--dimensional counterpart.  This way, a  hyperbolic field theory can be viewed as an infinite dimensional Hamiltonian system, together with its Poisson brackets, canonical transformations and even some analog of the Hamilton--Jacobi theory (see Dedecker \cite{Ded53,MR748931,MR761267}). To describe the invariance of its structure upon the choice of the Cauchy surface (initial--value--surface) in space--time, the notion of a {\em multi--symplectic geometry} was introduced by one of us (JK) (see, e.g., \cite{multi,multi1}), which proved to be an effective tool. It was later used in different contexts by many authors \cite{GIM,GIM2}.\par
All the \virg{symplectic} results within the multi--symplectic approach are obtained {\em via} ``integration by parts'', where the surface integrals are assumed to vanish {\em a priori} due to ``appropriate fall-off conditions at infinity''. This makes the multi--symplectic approach conceptually inadequate in those contexts, like,  e.g., General Relativity, where all the volume integrals, being gauge--dependent, have no physical significance, and the  only meaningful information is carried by the boundary integrals, like the  A.D.M.~energy (cf.~\cite{aff}). But also in special--relativistic field theories boundary integrals play an important role: the field energy (Hamiltonian) within a finite volume $V$ cannot be uniquely defined unless we specify boundary conditions on its boundary $\partial V$. Different Hamiltonians differ by a surface integral.\par
It turns out that all these drawbacks can be eliminated if we observe that the entire multi--symplectic structure constitutes merely a particular aspect of a much richer structure which is: 1) truly {\em symplectic} and 2) finite--dimensional. Some aspects of this structure were already noticed  in \cite{multi} and \cite{multi1}, but the complete formulation of the theory was given in \cite{KT79}. The {\em multi--symplectic} aspect is completely covered by the so called \virg{canonical Poincar\'e--Cartan form} of the theory.\par
This symplectic theory was based on the observation  that the space ${\cal P}_x^I$ of jets of sections of the Hamiltonian field theory carries a {\em canonical} symplectic structure  at each space--time point $x \in M$ separately, whereas the jets of the solutions of the field equations fill up   its Lagrangian submanifold ${\cal D}_x \subset {\cal P}_x^I$.
More precisely, systems of (nonlinear) partial differential equations (PDE) on a base manifold $M$ can be always considered as a collection ${\cal D} \subset J^k {\cal P}$ of ``admissible jets'' of sections of a certain bundle ${\cal P} \rightarrow M$.
In such a perspective, a section $s$ of ${\cal P}$ satisfies the given PDE at a point $x \in M$ if and only if $j^k s (x) \in {\cal D}$.
Our theory applies to a special class of PDEs, where every fibre $J_x^1 {\cal P}$ carries a natural (canonical) symplectic (or pre--symplectic) structure and the dynamics ${\cal D}_x$ constitutes  a Lagrangian (i.e.,~maximal, isotropic) submanifold of the fibre. It is well--known that this class contains the (systems of)  Euler--Lagrange equations associated with a first--order variational principle    \cite{KT79}, but it also goes   beyond  the  mere  calculus of variations,  encompassing cases admitting multiple variational principles, none of   them being more fundamental than the other ones. For instance, in the theory of   General Relativity,   at least four different variational approaches have been proposed \cite{aff,aff2,aff1}, based on:
\begin{enumerate}
\item  the  Hilbert Lagrangian, which depends upon the metric and its derivatives up to the \emph{second order}, \item  the  Einstein Lagrangian, which depends upon the metric and its \emph{first--order} derivatives, \item  the  Palatini Lagrangian, which depends upon \emph{both} the metric \emph{and} the connection, treated {\em a priori} as independent variables, together with their first--order derivatives,    \item     the ``affine'' Lagrangian, which depends  upon the \emph{connection and its first--order derivatives} contained in the corresponding curvature tensor.
\end{enumerate}
It turns out that the phase bundle of the theory remains the same, namely, the tensor product of  the bundle of metrics  by the bundle of connections over the space--time $M$, no matter which variational formulation is chosen. Passing from one formulation to another, the role of positions and momenta is interchanged, but the symplectic structure of the theory remains unchanged. This appears to be a general rule: except for some simple and academic examples, the  particular variational principle used to derive the field equations plays no role: what counts are the field equations, together with the underlying symplectic structure. In particular, ``adding a complete divergence'' (even of an arbitrary high order) to the Lagrangian does not change this structure. Even if at the beginning additional (artificial) momenta arise, the resulting symplectic space is degenerate. We show in Section \ref{secModLagrTotDiv} that the quotient space with respect to this degeneracy is equal to the original symplectic structure.

\par

In the present paper we give a complete description of the symplectic structures related with higher--order Lagrangians. However, by an appropriate choice of variables, any PDE system can always be reformulated as a first order system, possibly with constraints. So, one may insist that the description given in \cite{KT79} is sufficient, since it   only remains to  handle appropriately those constraints.  It should be stressed, however, that  these constraints are very special and no adequate symplectic treatment has been proposed in the literature (on this concern, see also \cite{Luca2011} and references therein). Consequently, canonical formulation of higher order Lagrangian theory was never formulated in a consistent way. To our best knowledge, the correct notion of the field energy and the construction of the Poincar\'e--Cartan form, which is obtained here as a simple corollary, was never done for higher--order variational problems.

\subsection{Structure of the paper}
In the preliminary Section \ref{SecToyModel} we briefly summarize the well--known construction of the \emph{infinitesimal phase bundle} $\mathcal{P}^I$ for a first--order variational problem, showing that it carries a canonical symplectic structure with respect to which the Euler--Lagrange equations become the generating \emph{formulae}  for a Lagrangian submanifold. Besides paving the way to the higher--order case, Section \ref{SecToyModel} will also serve the purpose of introducing and explaining the main notations and conventions used throughout the whole paper. In Section \ref{expansion} we briefly sketch the role of constraints, both the {\em momentum constraints} and the {\em Lagrangian constraints}, and show how to remove the irrelevant degrees of freedom of the theory {\em via} its symplectic reduction. This ``philosophy'' is employed in Section \ref{SecHiOrd}, where a theory with higher--order Lagrangian is treated as a first--order theory with Lagrangian constraints. Such an approach simplifies considerably the theory and allows us to go over the same steps presented in Section \ref{SecToyModel}. At the end, however, an additional symplectic reduction with respect to the degeneracy implied by the constraints is necessary.
Section \ref{SecByGiovanni} is entirely dedicated to the proof of the equivalence between the infinitesimal phase bundle and, so to speak, the ``space of vertical differentials of Lagrangian densities'', which is the key result used in the preceding Sections   \ref{SecToyModel}  and \ref{SecHiOrd}.
In Section \ref{SecPoinCart} we go over the classical notion of Poincar\'e--Cartan form for first--order Lagrangian field theories. In particular, we give a simple example illustrating a misunderstanding concerning the invariance of the Poincar\'e--Cartan form with respect to Galileian transformations in Mechanics, which has frequently  led  to serious errors in  papers dealing with canonical field theory. The ultimate goal of   Section \ref{SecPoinCart} is the higher--order version of the Poincar\'e--Cartan theory.  The proof of the identification $VJ^1=J^1V$ is carried out in the language of nonlinear differential operators between fibre bundles (see Section \ref{appVJ1UgJ1V}). Being very technical, it was moved to the Appendix. Finally, we add a handy list of symbols (see Section \ref{subListSymb}) to help the reader to keep track of the many objects involved.

We stress that, in order to  perform important steps like, e.g.,  defining gravitational energy or   preparing general relativity theory for quantization, physicists   \emph{must} rely on coordinate calculations. This is the  reason why all the formal definitions and theorems presented here are complemented by a detailed coordinate descriptions of the structures involved.

\section{First--order variational  problems: a remainder}\label{SecToyModel}
Here, as everywhere else in the paper, $M$ denotes the $n$--dimensional manifold of independent variables (i.e., the space--time underlying the theory) and $x^\mu $ are its coordinates.\par
To deal with a first--order variational problem we need a   \emph{configuration   bundle}, i.e., a bundle  $\pi:\Q\to M$, whose fibre $\Q_x$ at $x\in M$ represents all possible values of the fields at $x$. Fibre coordinates of $\Q$ are denoted by $  \varphi^K$, $K=1,\dots , N$, where $N$ denotes the number of dependent variables of the theory.
In the next Sections \ref{ToyStep1},  \ref{ToyStep2} and  \ref{ToyStep3} we carry out a parallel construction to the one which leads to $TT^\ast \Q$ in Mechanics, proving, in Theorem \ref{ThToyModel}, that the result  is the same as ``going the other way'', i.e., leading to  $T^\ast T\Q$ instead. Much as in Mechanics, the Euler--Lagrange equations, together with the definition of the canonical momenta, will take the shape of the generating \emph{formulae}  for a Lagrangian submanifold. It should be stressed that, in spite of the evident parallelism with Mechanics, here all the canonical forms are, by their nature, \emph{vector--density--valued}.
\subsection{Vector densities on $M$}
Differential $(n-1)$--forms on $M$ are (smooth) sections of the bundle of $(n-1)$--covectors on $M$, henceforth denoted by
\begin{equation}\label{eqLambdaEnneMenoUno}
\Lambda^{n-1} M := \bigwedge^{n-1}\left( T^* M \right)=\underset{n-1\textrm{ times}}{\underbrace{T^\ast M\wedge\cdots\wedge T^\ast M}}\, .
\end{equation}
Similarly, the bundle of volume forms ($n$--forms) on $M$ is denoted by $\Lambda^{n} M $.
To write down   sections of $\Lambda^{n-1} M $ and $\Lambda^{n} M $ in local coordinates, we use the following symbols:
\begin{eqnarray}
   \partial_\mu &:=& \frac{\partial}{\partial x^\mu} \ ,\nonumber\\
    \mbox{\rm d}^n x &:=& \mbox{\rm d}  x^1 \wedge \cdots
    \wedge \mbox{\rm d}  x^n \ , \nonumber\\
    \dens{\mu} &=& 
     (-1)^\mu\mbox{\rm d}  x^1 \wedge \cdots \wedge {\rm d}  x^{\mu-1}\wedge {\rm d}  x^{\mu+1}\wedge\cdots
    \wedge \mbox{\rm d}  x^n  \ .\label{defDEnneMenoUnoICS}
\end{eqnarray}
So, {\em modulo} obvious considerations concerning orientation (internal {\em versus} external, see, e.g.,~\cite{CJK}),   sections of \eqref{eqLambdaEnneMenoUno} can be safely  identified with {\em vector densities} on $M$, whereas sections of $\Lambda^{n} M $ will be refered to as {\em scalar densities}.
\subsection{Phase bundle and   canonical forms}\label{ToyStep1}
%
%
The next step is to introduce the field momenta at $x\in M$. They are described by  vector--density--valued covectors on $\Q_x$, i.e.,~elements of the space
\begin{equation}\label{eqDefPhaseBundlePunto}
  {\cal P}_x :=   T^* \Q_x  \otimes_\R \bigwedge^{n-1}
 T^*_x M \ .
\end{equation}
The collection of all these spaces forms the bundle
\begin{equation}\label{eqDefPhaseBundle}
  {\cal P} :=  V^*\Q  \otimes_\Q \bigwedge^{n-1} M\ ,
\end{equation}
where $V\Q$ is the \emph{vertical bundle} of $\Q$: $(V\Q)_q:=T_q\Q_{\pi(q)}$, for all $q\in \Q$, and $V^*\Q$ is its dual.
Indeed, in view of \eqref{eqLambdaEnneMenoUno}, the   bundle $\mathcal{P}$   defined by \eqref{eqDefPhaseBundle} can be restricted to the fibre $\Q_x$ of $\Q$ over $x$, yielding ${\cal P}_x :={\cal P}|_{\Q_x}$.

Hence, ${\cal P}$ can be considered both as a bundle over $\Q$ (according to \eqref{eqDefPhaseBundle}) and as a bundle over $M$ (according to \eqref{eqDefPhaseBundlePunto}). However, in the first case, it is linear, and $\mathrm{rank}\,{\cal P}=\dim \Q+\dim M$, whereas in the second case, it is, in general, nonlinear, and $\mathrm{rank}\,{\cal P}=2\dim \Q+\dim M$.
\begin{definition}
 The bundle $\mathcal{P}$ defined by \eqref{eqDefPhaseBundle} is called the \emph{phase bundle}.
\end{definition}

The basic tool of our construction is  the notion of the \emph{vertical differential}, which is denoted by $\delta$ and corresponds to what is called the \virg{variation} in the classical calculus of variations (see, e.g., \cite{MR1368401,MR2004181}).

\begin{definition}
Given a bundle $B$ over $M$, the vertical differential $\delta$ is the restriction of the external derivative operator $d$, defined on the bundle manifold $B$, to each fibre $B_x$ separately, for all $x\in M$.
\end{definition}

When applied to a function $f$ on $B$, the value of $\delta f$ is, therefore, an element of the space $V^*B$. The notion of vertical derivative will be used for both the configuration bundle $\pi : \Q \rightarrow M$ and the phase bundle ${\cal P}$, and not only for the functions but also for the differential forms living on these spaces.

In particular, the \virg{vertical forms} $\delta \varphi^K $ provide a base of $V^*\Q $. Hence, using the base elements \eqref{defDEnneMenoUnoICS} and the fibre coordinates  $\varphi^K$ of $\Q$, the generic section $p$ of $\mathcal{P}$ can be written in coordinates as
\[
    p = p^\mu_{\ K} \delta \varphi^K \otimes \dens{\mu} \ .
\]
When it does not lead to any misunderstanding, we may skip the index $K$, which  labels the unknown fields $\varphi^K$ and the corresponding components $p^\mu_K$ of the momenta, so that the last expression reads
\begin{equation}\label{typical-p}
    p 
    =p^\mu \delta \varphi \otimes \dens{\mu} \ .
\end{equation}
Recall that, by its definition \eqref{eqDefPhaseBundle}, the phase bundle ${\cal P}$ is a bundle over the configuration bundle, i.e., there is a projection $\Pi: \P \rightarrow Q$. So, any (vector--density--valued) covector $p$ on $\Q$ can be pulled back to $\P$. This way we obtain a canonical (vector--density--valued) one--form on $\P$, defined by
\[
    \theta_p := \Pi^* p, \quad p\in \P \ .
\]
In   coordinates, $\theta$ can be written (with an obvious abuse of language) by the same formula \eqref{typical-p}:
\begin{equation}
  \theta= p^\mu \delta \varphi \otimes
\dens{\mu}\ . \label{theta}
\end{equation}
(The obvious difference between these formulae is that $\delta \varphi$ denotes in \eqref{typical-p} a covector on $\Q$, whereas in \eqref{theta} it denotes its pull--back to $\P$.)
Its vertical differential
\begin{equation}
  \omega := \delta \theta  =
  \left( \delta p^\mu \wedge \delta \varphi \right) \otimes
\dens{\mu}  \label{omega}
\end{equation}
is the progenitor of the symplectic form we are trying to construct.

\begin{remark}
The form \eqref{omega} is symplectic only when $n-1= 0$, i.e., when the basis manifold is one--dimensional. This is the case of Mechanics, when $M$ stands for the time axis. In generic case there are $n$ times more momenta $p^\mu_{\ K}$ than configurations $\varphi^K$ and no symplectic structure arises at this stage.
\end{remark}

\begin{remark}
    Observe that the vector--density component ``$\dens{\mu}$'' of \eqref{theta} behaves as a constant under the action of the vertical differential $\delta$. It is constant, indeed, along every fibre $\P_x$.
\end{remark}
\subsection{Jet--extensions of the above structure}\label{ToyStep2}

For the sections of the bundles over $M$ (like $\Q$ or $\P$), as well as    for the objects living on these sections, also the so--called \emph{space--time  differential} ``${\rm d}$'' can be defined, which is nothing but the first jet--extension of the exterior derivative in the space--time $M$.
To define this operator, we concentrate on the last factors of the forms \eqref{theta} and \eqref{omega} and treat them as differential $(n-1)$--forms on the  space--time $M$, with the first factors playing the role of coefficients. The space--time exterior derivative of these forms produces   $n$--forms. For this purpose the derivatives of the coefficients are necessary, and this is why these $n$--forms do not live on the bundles $\Q$ and ${\cal P}$, but on their first jet--extensions. Hence, we introduce the {\em infinitesimal configuration space}
\begin{equation}
{\cal Q}^I:=J^1\Q\,  ,\label{eqInfConfBund}
\end{equation}
(first jet--extension of $\Q$) and the first jet--extension
\[
\widetilde{\cal P}^I := J^1{\cal P}
\]
of the bundle ${\cal P}$  over $M$. The latter has, at the moment, no specific name (for reasons which will be obvious later). Of course, $\widetilde{\cal P}^I$ can also be treated as a bundle over $\Q^I$, where the projection is given by the first jet--extension $j^1(\Pi)$ of $\Pi$, i.e.,
\[
    j^1(\Pi) : J^1{\cal P} \rightarrow J^1\Q \ .
\]

\begin{remark}\label{remPhaseBundlefibrereOnQI}
   Any system $(x^\mu, \varphi^K)$ of coordinates in $\Q$ gives rise to a system of coordinates $(x^\mu, \varphi^K ,   \varphi^K_{,\mu } )$ on ${\cal Q}^I$, where we use the jet--adapted notation $\varphi^K_{,\mu }:= \partial_\mu \varphi^K$. Again, the index $K$ labeling the field degrees of freedom can be safely skipped in most \emph{formulae}, which simplifies considerably our notation.
   Similarly, the coordinates $(x^\mu, \varphi ,  p^\mu )$ in ${\cal P}$ give rise to the coordinates
   $(x^\mu, \varphi ,  p^\mu ,\varphi_{,\mu } , p^\mu_{\ ,\nu})$ on its first jet--extension $\widetilde{\cal P}^I $, where $p^\mu_{\ ,\nu} := \partial_\nu p^\mu$. (Remember that every $\varphi$, together with its jets, has an additional index $K$ upstairs, whereas every $p$, together with its jets, has an additional index $K$ downstairs!) These coordinates are compatible with the bundle projection, i.e.,~the projection $j^1(\Pi)$ form $\widetilde{\cal P}^I$ to ${\cal Q}^I$ consists in simply forgetting the coordinates $p^\mu$ and $ p^\mu_{\ ,\nu}$.
\end{remark}

By acting on the (vector--density--valued) one--form  \eqref{theta}, the space--time differential \virg{${\rm d}$}  produces a (scalar--density--valued) one--form on $\widetilde{\cal P}^I$, which will be denoted by
\begin{equation}\label{thetaITilde}
    \widetilde{\theta}^I:=\mbox{\rm d}\theta = \partial_\mu \left( p^\mu \delta \varphi \right) \otimes
    \left({\rm d}x^\mu \wedge  \dens{\mu}  \right) = \partial_\mu \left( p^\mu \delta \varphi \right) \otimes
     \mbox{\rm d}^n x   \ .
\end{equation}
The same procedure, applied to the (vector--density--valued) two--form \eqref{omega}, yields a (scalar--density--valued) two--form on $\widetilde{\cal P}^I$, denoted by
\begin{equation}\label{omegaITilde}
    \widetilde{\omega}^I:=\mbox{\rm d}\omega  = \partial_\mu\left( \delta p^\mu \wedge
    \delta \varphi \right) \otimes
    \left({\rm d}x^\mu \wedge  \dens{\mu}  \right) =
    \left( \delta p^\mu_{\ ,\mu} \wedge
    \delta \varphi  + \delta p^\mu \wedge
    \delta \varphi_{,\mu }\right) \otimes
     \mbox{\rm d}^n x  \ ,
\end{equation}
where we use the jet--adapted notation $\varphi_{,\mu }:= \partial_\mu \varphi$ and $p^\mu_{\ ,\nu} := \partial_\nu p^\mu$, discussed in Remark \ref{remPhaseBundlefibrereOnQI}. Observe that $   \widetilde{\omega}^I$
  is a  scalar--density--valued  two--form on   $\widetilde{\cal P}^I$. Of course, the vertical exterior derivative ``$\delta$'' and the jet--extension ``${\rm d}$'' of the space--time exterior derivative do commute, because they differentiate with respect to different variables. For this reason, we have
\begin{equation}\label{degene}
 \widetilde{\omega}^I= {\rm d} \omega = {\rm d} \delta \theta =
 \delta {\rm d} \theta
 =\delta     \widetilde{\theta}^I.
\end{equation}
\subsection{The infinitesimal phase bundle}\label{ToyStep3}
The pre--symplectic form $\widetilde{\omega}^I$ defined by \eqref{omegaITilde} is obviously degenerate: it does not depend upon all the jet coordinates $p^\mu_{\ ,\nu}$, but only on their trace $p^\mu_{\ ,\mu}$.  The degeneracy distribution of $\widetilde{\omega}^I$  corresponds precisely to the foliation of $\widetilde{\cal P}^I$ with respect to the following equivalence relation: two elements of $\widetilde{\cal P}^I$ are equivalent if and only if they have the same value of the coordinates $(x^\mu, \varphi ,  p^\mu ,\varphi_{,\mu } , p^\mu_{\ ,\mu})$ (for an obvious coordinate--independent definition of this foliation see \cite{KT79}). Hence, the form $\widetilde{\omega}^I$ defines a non--degenerate (volume--form--valued) two--form  $\omega^I$ on the quotient space
\begin{equation}\label{defInfPhasBund}
    {\cal P}^I :=\frac{ \widetilde{\cal P}^I }{{\textrm{degeneracy of }}\, \widetilde{\omega}^I} \ .
\end{equation}
A convenient choice of local coordinates for $  {\cal P}^I $ requires the \emph{current}
\begin{equation}
j:=p^\mu_{\ ,\mu}\, .
\end{equation}
Indeed,  we obtain a coordinate system on $  {\cal P}^I$:
\begin{equation}\label{coordinatePI}
{\cal P}^I\equiv \{(x^\mu, \varphi ,  p^\mu ,\varphi_{,\mu } , j)\}\, .
\end{equation}
In the coordinates \eqref{coordinatePI}, the non--degeneracy of the reduced two--form
\begin{equation}\label{eqDefOmegaI}
    \omega^I = \left(  \delta  j \wedge \delta \varphi
+ \delta p^\mu \wedge \delta   \varphi_{, \mu}
 \right) \otimes  \mbox{\rm d}^n x  \ ,
\end{equation}
  becomes evident and, consequently, also the fact  that the form $\omega^I$ is symplectic. More precisely, we obtain a family of symplectic structures, each one corresponding to a specific choice of the volume element in $M$, which must be paired with the last ingredient ``${\rm d}^n x$'' of the form \eqref{eqDefOmegaI} in order to produce  a number. Since all these symplectic forms are proportional to each other,   the notion of a Lagrangian (i.e.,~maximal, isotropic) submanifold is common to all of them.
Notice that the degeneracy distribution of \eqref{omegaITilde} can be projected on $\Q^I$, so that ${\cal P}^I$ inherits the structure of a bundle over $\Q^I$ from $ \widetilde{\cal P}^I$. Obviously, ${\cal P}^I$ is also a bundle over $\Q$ and $M$ (see diagram \eqref{eqDaDimostrareMODIF1} below). We stress again that \eqref{eqDefOmegaI} is just  a convenient coordinate definition: the form $   \omega^I$ has the same intrinsic character as its precursor $\omega$.
\begin{definition}
 The space ${\cal P}^I$ defined by \eqref{defInfPhasBund} is called the \emph{infinitesimal phase bundle}. The (scalar--density--valued) two--form  $  \omega^I$ defined by \eqref{eqDefOmegaI} is its \emph{infinitesimal symplectic structure}.
\end{definition}
\begin{equation}\label{eqDaDimostrareMODIF1}
\xymatrix{
\widetilde{\P}^I=J^1{\P}\ar@{->>}[rr]^{}\ar[dr]\ar[d]&&  \P^I\ar[dl]\\
\P\ar[dr]^\Pi&\Q^I \ar[d] &\\
&\Q\ar[d]^\pi &\\
&M &
}
\end{equation}
Observe that also the one--form $\widetilde{\theta}^I$ defined by  \eqref{thetaITilde} is compatible with the degeneracy distribution of $ \widetilde{\omega}^I$, so that it also descends to a one--form
\begin{equation}\label{thetaISenzaTilde}
     {\theta}^I =\left( j  \delta \varphi
+  p^\mu  \delta   \varphi_{, \mu}\right)\otimes {\rm d}^n x \
\end{equation}
on ${\cal P}^I $. Moreover,
\begin{equation}\label{OmegaIDifferenzialeDiThetaI}
  \omega^I=\delta  {\theta}^I.
\end{equation}

The canonical approach to field theory is based on the following, fundamental theorem:
\begin{theorem}\label{ThToyModel}
 There is a canonical identification
 \begin{equation}\label{eqPrimaIdentificazioneCanonica}
{\cal P}^I  \cong  V^*{\cal Q}^I
    \otimes_{{\cal Q}^I} \Lambda^{n}
    M
\end{equation}
of linear bundles over ${\cal Q}^I$.
\end{theorem}
\begin{proof}
Carried out in Section \ref{SecByGiovanni}.
\end{proof}
By the same reasons behind the equivalence of \eqref{eqDefPhaseBundle} and  \eqref{eqDefPhaseBundlePunto}, the unique identification \eqref{eqPrimaIdentificazioneCanonica} of  linear bundles over ${\cal Q}^I$ can be seen as a family of identifications of (usually nonlinear) bundles over ${\cal Q}^I_x$, namely
\begin{equation}\label{eqPrimaIdentificazioneCanonica-at-x}
    {\cal P}^I_x \cong  T^*{\cal Q}^I_x
    \otimes_\R \bigwedge^{n}
     T^*_x M \ ,\quad\forall x\in M\ .
\end{equation}
\subsection{Euler--Lagrange equations as a Lagrangian submanifold}\label{subELdynamics}
In the present framework, a Lagrangian density
\begin{equation}\label{lagr}
    {\cal L} = L\, \mbox{\rm d}^n x \
\end{equation}
is treated as a (scalar--density--valued) zero--form on the infinitesimal configuration bundle ${\cal Q}^I$. Its vertical differential becomes a (scalar--density--valued) covector on ${\cal Q}^I$:
\begin{equation}\label{eqVertDiffLagr}
  \delta {\cal L}  \in  V^*{\cal Q}^I  \otimes_{{\cal Q}^I} \Lambda^{n}
    M \,   .
\end{equation}
For every space--time point $x\in M$, the collection of all these covectors (i.e.,~the graph of $\delta {\cal L}_x$)  is a sumbanifold ${\cal D}_x$ of the infinitesimal phase bundle \eqref{eqPrimaIdentificazioneCanonica-at-x}.

In accordance with our point of view on PDEs, the submanifold ${\cal D}$ can be treated as ``the equation'' generated by ${\cal L}$, i.e.,~the space of ``admissible jets'' of sections of the phase bundle ${\cal P}$. More precisely, we say that a section $s$ of the phase bundle satisfies our PDE if and only if the equivalence class $[j^1(s)]$ of its first jet $j^1(s)$ belongs to ${\cal D}$. With an obvious abuse of language, we shall write $j^1(s) \in {\cal D}$.

\begin{corollary}\label{corELeqsAreLagrSubs}
For any $x\in M$, the submanifold
\begin{equation}\label{defDynamics}
{\cal D}_x:= {\rm graph} (\delta {\cal L}_x) \subset \P^I_x
\end{equation}
is Lagrangian with respect to $ \omega^I_x$ and, as a PDE, it is equivalent to the Euler--Lagrange equations associated with ${\cal L} $, together with a definition of the corresponding canonical momenta.
\end{corollary}
\begin{proof}
By the very definition \eqref{defDynamics} of ${\cal D}$, the equality
\begin{equation}\label{eqGenFormula}
    \delta {\cal L} = \left. \theta^I \right|_{\cal D} \ ,
\end{equation}
which is referred to as the \emph{generating formula} for ${\cal D}$, is satisfied  on the submanifold ${\cal D}$, so that \eqref{OmegaIDifferenzialeDiThetaI} implies  $\left. \omega^I\right|_{\cal D}= \delta \left. \theta^I\right| = \delta \delta {\cal L} \equiv 0$, i.e., ${\cal D}$ is Lagrangian.\par
When writing  equation  \eqref{eqGenFormula} in local coordinates, we can skip the volume form ${\rm d}^n x$, which is present in both \eqref{lagr} and \eqref{thetaISenzaTilde}. This way we obtain
\begin{equation}\label{eqQuasiEulLag}
    \delta L (\varphi, \varphi_{,\mu}) = j  \delta \varphi
+  p^\mu  \delta   \varphi_{, \mu} \ .
\end{equation}
This is equivalent to the first--order PDE
\begin{eqnarray}
   p^\mu &=& \frac{\partial L}{\partial \varphi_{, \mu}} \ ,\label{defMom} \\
  j = \partial_\mu p^\mu &=& \frac{\partial L}{\partial \varphi}\ , \label{eqEL} \,
\end{eqnarray}
which, in turn, is equivalent to the (second--order) Euler--Lagrange equations.  Indeed, treating \eqref{defMom} as the  definition of the ``auxiliary variables'', namely the canonical momenta $p^\mu$, and  plugging them into \eqref{eqEL}, we obtain the Euler--Lagrange system.\par
\
\end{proof}
\begin{remark}\label{control-mode}
   The following terminology, taken from the control theory, simplifies considerably the description of various physical phenomena in terms of symplectic geometry. Namely, let $P$ be the symplectic space  describing a physical system. Whenever $P$ is represented as the   co--tangent bundle of a certain manifold $Q$, i.e.,~$P=T^*Q$, we   call this representation a   ``control mode''.  In this perspective, coordinates $q^i$ on $Q$ become the ``control parameters'' and the corresponding momenta $p_i$ the ``response parameters''. The condition imposed on the admissible states of the system, i.e.,~$(q^i, p_i) \in D\subset P$, where $D$ is  a  Lagrangian submanifold $D\subset P$, captures the physical laws governing the system. It 
is often interpreted as a condition imposed on the momenta when the positions are given. We call it ``a control--response relation''. Hence, we may say that the Euler--Lagrande equations \eqref{eqGenFormula} provide a control--response relation in the ``Lagrangian'' control mode \eqref{eqPrimaIdentificazioneCanonica-at-x}.
\end{remark}

\section{Expansion and reduction}\label{expansion}

It may happen that, when constructing a mathematical model of a given physical phenomenon, we take into account an additional field variable, say $\psi$, which later may prove itself to be irrelevant. At the beginning, we just add the new degree of freedom $\varphi^{N+1} := \psi$ to the previous $N$ fields $\phi^K$'s. Consequently, the new momentum $p_{N+1}^\mu =: r^\mu$, canonically conjugate to $\varphi^{N+1}$, and the new current $i:= r^\mu_\mu$ arise.  Doing so, the phase bundle gets new dimensions, the canonical symplectic form \eqref{eqDefOmegaI} acquires new terms and the field equations \eqref{eqQuasiEulLag} (or, equivalently, \eqref{defMom}--\eqref{eqEL}) are supplemented by new ones:
\begin{eqnarray}
   r^{\mu} &=& \frac{\partial L}{\partial \phi_{, \mu}} \ ,\label{defMom-N} \\
  i = \partial_\mu r^\mu &=& \frac{\partial L}{\partial \phi}\ . \label{eqEL-N} \,
\end{eqnarray}
Suppose now that the Lagrangian of the theory does not depend upon the values of the new variables. This means that they are irrelevant for the phenomenon we are modelling. In such a case the right--hand sides of \eqref{defMom-N}--\eqref{eqEL-N} vanish identically. This fact can be treated as the manifestation of the additional constraints  $r^\mu \equiv 0$ and $i \equiv 0$. But, when restricted to the subspace of points satisfying the constraints, the symplectic form is no longer non--degenerate.  So, removing the irrelevant variables requires a symplectic reduction with respect to the degeneracy of the symplectic form. This means that we identify the states which differ by the values of the irrelevant variables only. The quotient space is isomorphic with the previous phase space ${\cal P}^I$.\par
Such a  scheme is quite general. Constraints imposed on the phase space can introduce a degeneracy of the symplectic form $\omega^I$. The leaves of this degeneracy describe  the ``irrelevant degrees of freedom'', which are often called ``gauge degrees of freedom''. We can remove the degeneracy if we pass to the quotient space with respect to this gauge. The resulting quotient space is parametrized by the ``gauge invariants'', i.e.,~by the quantities which do not depend upon the gauges. The resulting symplectic form is non--degenerate. These techniques have already been used in \eqref{defInfPhasBund} to remove the degeneracy of the form \eqref{degene}.\par
If the constraints apply to the momenta $p^\mu$ and $j$ (``momentum constraints''), the corresponding gauge applies to the configurations $\varphi$ and $\varphi_{,\mu}$, like in the trivial example above. Of course, the gauge can be highly non--trivial if the constraints are non--linear.\par
But the opposite situation, namely when the constraints imposed on the configurations imply gauge in momenta, often happens. We call such constraints the ``Lagrangian constraints''. Suppose, therefore, that the admissible configurations of the theory are subject to the constraint equations
\begin{equation}\label{constr}
    C_a(\varphi, \varphi_{,\mu})= 0 \ , \ \ a = 1,\dots , k \ .
\end{equation}
We assume that these constraints are \emph{regular}, i.e.,~that the $k$ equations  \eqref{constr} define a submanifold of codimension $k$, and denote by ${\cal C} \subset \Q^I$ this ``constraint submanifold'', i.e.,~the collection of points satisfying these equations. If the Lagrangian density ${\cal L}$ is defined on the constraint sumbanifold ${\cal C}$ only, then equation \eqref{eqQuasiEulLag} is, {\em a priori}, meaningless, because the differential $\delta{\cal L}$ is not defined. More precisely, it is not {\em uniquely} defined, because we can use any extension $\widetilde{\cal L}$ to a neighborhood of ${\cal C}$, of the function ${\cal L}$, and take its differential $\delta \widetilde{\cal L}$ as a representative of $\delta {\cal L}$. This representation is, of course, not unique.\par
There is a strategy to simplify the generating formula \eqref{eqQuasiEulLag} as much as possible, in such a way that all the \emph{formulae} look the same in both  in the constrained case and in the constraint--free case. It consists in defining the differential $\delta {\cal L}$ as the {\em collection} of all possible covectors $\delta \widetilde{\cal L}$ obtained in this way. In other words,  $\delta {\cal L}$ {\em is not} a single (scalar--density--valued) covector on ${\cal Q}^I_x$, but rather the {\em collection} of all the covectors on $\Q^I$ which agree with the differential of ${\cal L}$ on  the constraints submanifold ${\cal C}$. Choosing a particular extension $\widetilde{\cal L}$ of ${\cal L}$, this collection can be described as
\begin{equation}\label{delta-tilde-L}
    \delta {\cal L} := \{ \delta \widetilde{\cal L} + \lambda^a \delta C_a \} \ ,
\end{equation}
where the ``Lagrange multipliers'' $\lambda^a$ assume all possible values. Formula \eqref{delta-tilde-L} shows that $\delta{\cal L}_x$ is not just a single covector on $\Q^I$, but rather  a $k$--parametric family of them. This means that the graph of $\delta{\cal L}$ has again the dimensionality of $\Q^I$: it is a sub--bundle of ${\cal P}$, whose basis ${\cal C}$ has co--dimension $k$ and whose fibres have dimension $k$. It is easy to see that, like in Corollary \ref{corELeqsAreLagrSubs}, the submanifold ${\cal D}_x:={\rm graph}(\delta{\cal L}_x)$ is again a Lagrangian submanifold of ${\cal P}_x$ and, as a PDE, it is equivalent to the Euler--Lagrange equations associated with ${\cal L} $ and the constraints \eqref{constr}, together with a ``definition of the corresponding canonical momenta''. Actually, the momenta are not uniquely defined in this case, but only up to a ``gauge'' described by \eqref{delta-tilde-L}.\par

Such a definition of $\delta {\cal L}$ highly simplifies  the notation. Indeed, field equations for a theory with Lagrangian constraints can again be written as \eqref{eqGenFormula} or, in coordinates, as \eqref{eqQuasiEulLag}. The response parameters on the right--hand side of \eqref{defMom} and \eqref{eqEL} are not given uniquely, but constitute  a family given by
\begin{eqnarray}
   p^\mu &=& \frac{\partial {\widetilde L}}{\partial \varphi_{, \mu}} +
    \lambda^a \frac{\partial C_a}{\partial \varphi_{, \mu}}\ ,\label{defMom-constr} \\
  j = \partial_\mu p^\mu &=& \frac{\partial {\widetilde L}}{\partial \varphi}+
    \lambda^a \frac{\partial C_a}{\partial \varphi}\ , \label{eqEL-constr} \,
\end{eqnarray}
where ${\widetilde L}$ is any restriction of $L$ to a neighbourhood of  the constraint submanifold ${\cal C}$. Our definition of $\delta {\cal L}$ allows us  to replace the last two \emph{formulae} by \eqref{eqQuasiEulLag}.

Being perfectly legal, the above formulation of a theory with Lagrangian constraints can be further simplified by removing the ``irrelevant degrees of freedom''. This redundancy is described not only by the entire class of jets possessing the same value of the trace $j = \partial_\mu p^\mu$, like in the unconstrained case, but also by the Lagrange multipliers $\lambda_a$. This alternative formulation of the theory consists in restricting the infinitesimal phase space bundle ${\cal P}^I$ to the submanifold ${\cal P}_{\cal C}^I$ composed of those fibres which  satisfy the constraints. The symplectic form $\omega^I$, restricted to $\C$, gives the form $\omega^I_{\cal C}$, which is degenerate. The  degeneracy foliation of $\omega^I_{\cal C}$ contains not only the complete degeneracy leaves of \eqref{defInfPhasBund}, but also the ``gauge leaves'' given by \eqref{delta-tilde-L}. The symplectic reduction consists in passing to the quotient space, where two states of the field are declared to be equivalent if they belong to the same leaf of the foliation. This means that two covectors on $\Q^I$ are equivalent if and only if they define the same (scalar--density--valued) covector on ${\cal C} \subset \Q^I$. As a result we obtain the {\em reduced} infinitesimal phase space
\begin{equation}\label{inf-red}
    {\cal P}^I_{\mathrm{reduced}} := \frac {{\cal P}_{\cal C}^I}{{\textrm{degeneracy of }}\, \omega^I_{\cal C}} \ .
\end{equation}
The above symplectic reduction plays role of the ultimate Ockham's Razor in our construction. Even if we begin our construction with too many parameters, the razor finally reduces it to the optimal shape.
\begin{theorem}\label{ThToyModel-red}
 There is a canonical identification
 \begin{equation}\label{eqPrimaIdentificazioneCanonica-red}
{\cal P}^I_{\mathrm{reduced}}  \cong \left(V^*{\cal C}  \right)
    \otimes_{{\cal C}} \Lambda^{n}
    M
\end{equation}
of linear bundles over ${\cal C}$.
\end{theorem}
\begin{proof}
Follows from Corollary \ref{corEgr} (see Section \ref{sec52} later on).
\end{proof}

\section{Higher order variational problems}\label{SecHiOrd}
Our construction of the symplectic framework for the calculus of variations, presented in Section \ref{SecToyModel}, was based on the following four steps.
\begin{enumerate}
\item We first define the phase bundle $\P$ of \virg{vector--density--valued} covectors on the fibres of the configuration bundle $\Q$. There are canonical forms $\omega$ and $\theta$ living on it (see Section \ref{ToyStep1}).
\item We take the first jet--extension $\widetilde{\P}^I$ of $\P$, together with the jet--extension (\virg{space--time derivatives}) $\widetilde{\omega}^I$ and $\widetilde{\theta}^I$ of the canonical forms (see Section \ref{ToyStep2}).
\item We observe that the canonical two--form $\widetilde{\omega}^I$ is degenerate and we define $ {\P}^I$ as the symplectic reduction of $\widetilde{\P}^I$ with respect to this degeneracy (see Section \ref{ToyStep3}).
\item We notice that the collection of all the jets satisfying the Euler--Lagrange equations for a given Lagrangian ${\cal L}$ corresponds to a Lagrangian submanifold ${\cal D}$ of the infinitesimal phase bundle $ {\P}^I$. This correspondence is accomplished  {\em via} the generating equation $\delta{\cal L}=\left. \theta\right|_{\cal D}$ (see Corollary \ref{corELeqsAreLagrSubs}).
\end{enumerate}

Its extension to  higher--order Lagrangians can be constructed in many equivalent ways. The \virg{royal road} which we use here consists in treating a $k\Th$ order variational problem as a first--order problem with Lagrangian constraints:
\begin{equation}\label{jets-in-jets}
   J^k \Phi \subset J^1(J^{k-1} \Phi )\ .
\end{equation}
This means that we first treat the space $J^{k-1} \Phi$ of $k-1\St$  jets of a given bundle $\Phi$ as the configuration bundle. The construction goes along the lines sketched above, but the final symplectic reduction with respect to the constraints \eqref{jets-in-jets} is necessary.

\subsection{The configuration space}\label{Step0HiOrd}
The space--time $M$ and its coordinates are the same as before. On the other hand, the role of $\Q$ is played now by the $k-1\St$ jet--extension
\begin{equation}\label{eqDefQ}
    \Q := J^{k-1} \Phi \
\end{equation}
of a   fibre bundle $\pi:\Phi\longrightarrow M$. The sections of $\Phi$ are the fields of the theory, and $\Phi_x$ is the space of all possible values of the fields at the point $x \in M$. As before,  we can skip the index $K$ labelling the fields and write $(x^\mu , \varphi)$ as coordinates on $\Phi$, instead of $(x^\mu , \varphi^K)$. Such an abuse of notation will simplify our job. The procedure  to recover the correct version of the \emph{formulae} which follow is simple: every $\varphi$ acquires an extra index $K$ upstairs, whereas every dual object (momenta and currents) acquires an extra index $K$ downstairs. \par
Accordingly, coordinates on $\Q$ are denoted by
\begin{equation}\label{eqDefQ_coord}
  (x^\mu , \varphi, \varphi_{\mu} , \varphi_{{\mu_1}{\mu_2}}, \dots
    , \varphi_{{\mu_1} \dots {\mu_{k-1}}}) \ ,
\end{equation}
where every coordinate $\varphi_{{\mu_1} \dots {\mu_i}}$ is  symmetric {\em a priori}. Equivalently, we can use the multi--index notation
\begin{equation}\label{eqCoordQI}
 (x^\mu ,  \varphi_{\mmu}) \ ,
\end{equation}
where $\mmu = (\mmu_1 , \mmu_2 , \dots , \mmu_n)$ is a multi--index. Its component $\mmu_i \in \mathbb{N}$, $1 \le i \le n$, tells us ``how many derivatives of $\varphi$ has been taken in the direction of the variable $x^i$ on M''. If $|\mmu|$ denotes the length of the multi--index, we have $0 \leq |\mmu| \leq (k-1)$.  \par
Now we can use the space $\Q$ as the starting point for the construction described in  Section \ref{SecToyModel}. Beware that, doing so, all the fibre  coordinates \eqref{eqCoordQI} will play the role of \emph{independent} field variables, so that, at the appropriate moment (Section \ref{StepEXTRAHiOrd} below), the additional relation (called {\em holonomy constraint}) must be imposed, to force the $\varphi_{\mmu}$'s to be the \emph{true} derivatives of $\varphi$.

\subsection{Phase bundle and   canonical forms on it}\label{Step1HiOrd}

The phase bundle ${\cal P}$ of the higher order theory is constructed in analogy with \eqref{eqDefPhaseBundlePunto}:
\begin{equation}\label{phase-bundle-higher}
  {\cal P}_x :=  T^* \Q_x \otimes_\R \bigwedge^{n-1}
 T^*_x M  ,\quad\forall x\in M \ .
\end{equation}
A typical element of $  {\cal P}$ reads
\begin{equation}\label{typic-p}
      p = \left( p^\lambda \delta \varphi +
  p^{\mu \lambda } \delta \varphi_\mu + \dots +
  p^{{\mu_1} \dots {\mu_{k-1}} \lambda } \delta \varphi_{{\mu_1} \dots {\mu_{k-1}}}
    \right) \otimes \dens{\lambda} \ ,
\end{equation}
where, for every $0\leq l \leq (k-1)$, the coefficients $p^{{\mu_1} \dots {\mu_{l}} \lambda }$ are {\em a priori} symmetric with respect to the indices $({\mu_1} \dots {\mu_{l}})$, viz.
\begin{equation}\label{sym-1}
    p^{{\mu_1} \dots {\mu_{l}} \lambda } =
    p^{({\mu_1} \dots {\mu_{l}}) \lambda } \ ,
\end{equation}
but no symmetry of the momenta $p^{{\mu_1} \dots {\mu_{l}} \lambda }$ with respect to the last index $\lambda$ is assumed.\par

We shall also need the sub--bundle ${\cal S} \subset {\cal P}$ consisting of totally symmetric momenta,  i.e.,
\begin{equation}\label{sym-2}
    p^{{\mu_1} \dots {\mu_{l}} \lambda } =
    p^{({\mu_1} \dots {\mu_{l}} \lambda )} \, ,
\end{equation}
whose intrinsic definition is put off in the Appendix \ref{app-momentum-gauge}. At this point we only mention that the momenta will always be applied to holonomic jets and, consequently, the non--symmetric part of the momentum will play role of a gauge parameter. For reason which will be clear in the sequel, it is useful to develop in parallel both versions of the theory: the non--symmetric one, based on the bundle ${\cal P}$ and the symmetric one, based on its sub--bundle ${\cal S}$.\par

In the multi--index notation, the element \eqref{typic-p} can be written as
\begin{equation}\label{typic-p-multi}
      p =  \sum_{|\mmu|\leq (k-1)} p^{\mmu \lambda} \delta \varphi_{\mmu}
    \otimes
\dens{\lambda}  \ ,
\end{equation}
where the summation runs over all multi--indices and over all $\lambda$'s.
Observe that the correspondence between the multi--index $\mmu = (\mmu_1 , \mmu_2 , \dots , \mmu_n)$ and the corresponding index  $(\mu_1 , \dots , \mu_l)$ implies:
\[
    p^{\mmu  \lambda} = l! \cdot p^{{\mu_1} \dots {\mu_{l}} \lambda } \ ,
\]
because, due to the symmetry, every term of the sum \eqref{typic-p-multi} represents $l!$ identical terms of the sum \eqref{typic-p}.
So far, besides an inevitable proliferation of indices, no critical differences with respect to the first--order case have yet been met. Also the definition of the canonical forms
\begin{eqnarray*}
  \theta &=& \left( p^\lambda \delta \varphi +
  p^{\mu \lambda } \delta \varphi_\mu + \dots +
  p^{{\mu_1} \dots {\mu_{k-1}} \lambda } \delta \varphi_{{\mu_1} \dots {\mu_{k-1}}}
    \right) \otimes
\dens{\lambda}, \\
  \omega = \delta \theta &=&
  \left( \delta p^\lambda \wedge \delta \varphi
  +
  \delta p^{\mu \lambda } \wedge \delta \varphi_\mu + \dots +
  \delta p^{{\mu_1} \dots {\mu_{k-1}} \lambda } \wedge
  \delta \varphi_{{\mu_1} \dots {\mu_{k-1}}}
  \right)\otimes
\dens{\lambda} \ ,
\end{eqnarray*}
is formally analogous to \eqref{theta} and \eqref{omega}, reading, in multi--index notation, respectively,
\begin{eqnarray*}
  \theta &=&  \sum_{|\mmu|\leq (k-1)} p^{\mmu \lambda} \delta \varphi_{\mmu}
      \otimes
  \dens{\lambda}  \ , \\
  \omega = \delta \theta &=&  \sum_{|\mmu|\leq (k-1)}
  \delta p^{\mmu \lambda} \wedge  \delta \varphi_{\mmu}
      \otimes
 \dens{\lambda}  \ .
\end{eqnarray*}

\subsection{Jet--extension of the phase bundle}\label{Step2HiOrd}
Much as in Section \ref{ToyStep2}, we produce now an ``oversized'' infinitesimal phase bundle, which later will be shrunk to appropriate proportions. The key difference with the first--order case is that the shrinking will be performed in two, conceptually separated, steps. This is the reason why the first--jet extension
\[
\widetilde{\widetilde{\cal P}}^I := J^1{\cal P}
\]
of the bundle ${\cal P}\longrightarrow M$ is decorated with a double tilde. In the symmetric version of the theory we put:
\[
\widetilde{\widetilde{\cal S}}^I := J^1{\cal S} \ .
\]
According, we shall have the ``double tilde'' versions of the forms \eqref{thetaITilde} and \eqref{omegaITilde}:
\begin{eqnarray*}
  \widetilde{\widetilde{\theta}}^I &:=& \mbox{\rm d}\theta = \partial_\lambda
  \left( p^\lambda \delta \varphi +
  p^{\mu \lambda } \delta \varphi_\mu + \dots +
  p^{{\mu_1} \dots {\mu_{k-1}} \lambda } \delta \varphi_{{\mu_1} \dots {\mu_{k-1}}}
    \right) \otimes
     \mbox{\rm d}^n x  \\
   &=& \left( p^\lambda_{\ ,\lambda} \delta \varphi
  + p^\lambda \delta \varphi_{ , \lambda} +
  p^{\mu \lambda }_{\ \ ,\lambda} \delta \varphi_\mu +
  p^{\mu \lambda } \delta \varphi_{\mu ,\lambda}
  \right. \\
   &+&
  \left.
   p^{{\mu_1} \dots {\mu_{k-1}} \lambda }_{\ \ \ \ \ \ \ \ \ \ ,\lambda} \delta \varphi_{{\mu_1} \dots {\mu_{k-1}}}
  + p^{{\mu_1} \dots {\mu_{k-1}} \lambda } \delta \varphi_{{\mu_1} \dots {\mu_{k-1}} , \lambda} \right)
  \otimes
     \mbox{\rm d}^n x   \ ,
\end{eqnarray*}
and
\begin{eqnarray*}
  \widetilde{\widetilde{\omega}}^I &:=& \mbox{\rm d}\omega  = \mbox{\rm d} \delta \theta
    = \delta \mbox{\rm d} \theta = \delta \widetilde{\theta}^I\\
   &=& \left( \delta p^\lambda_{\ ,\lambda} \wedge \delta \varphi
  + \delta p^\lambda \wedge \delta \varphi_{ , \lambda} +
  \delta p^{\mu \lambda }_{\ \ ,\lambda} \wedge \delta \varphi_\mu +
  \delta p^{\mu \lambda } \wedge \delta \varphi_{\mu ,\lambda}
  \right. \\
  &+&
  \left.
   \delta p^{{\mu_1} \dots {\mu_{k-1}} \lambda }_{\ \ \ \ \ \ \ \ \ \ ,\lambda} \wedge \delta \varphi_{{\mu_1} \dots {\mu_{k-1}}}
  + \delta p^{{\mu_1} \dots {\mu_{k-1}} \lambda }
  \wedge \delta \varphi_{{\mu_1} \dots {\mu_{k-1}} , \lambda} \right)
  \otimes
     \mbox{\rm d}^n x   \ .
\end{eqnarray*}
In multi--index notation, above forms become, respectively,
\begin{eqnarray}
  \widetilde{\widetilde{\theta}}^I &=& 
  \partial_\lambda
  \left( \sum_{|\mmu|\leq (k-1)} p^{\mmu \lambda} \delta \varphi_{\mmu}
    \right) \otimes
     \mbox{\rm d}^n x  \label{eqThetaTildeTilde}\\
   &=& \sum_{|\mmu|\leq (k-1)}\left(
    p^{\mmu \lambda} \delta \varphi_{\mmu ,\lambda}
    + p^{\mmu \lambda}_{\ \ ,\lambda}
    \delta \varphi_{\mmu}
    \right)
  \otimes
     \mbox{\rm d}^n x   \ ,\label{eqOmegaTildeTilde}
\end{eqnarray}
and
\begin{eqnarray*}
  \widetilde{\widetilde{\omega}}^I 
   &=& \sum_{|\mmu|\leq (k-1)} \left(
    \delta p^{\mmu \lambda} \wedge \delta \varphi_{\mmu ,\lambda}
    + \delta p^{\mmu \lambda}_{\ \ ,\lambda} \wedge
    \delta \varphi_{\mmu} \right)
  \otimes
     \mbox{\rm d}^n x   \ .
\end{eqnarray*}
Much as   ${\widetilde{\cal P}}^I$ was a bundle over the ``infinitesimal configuration bundle'' ${\cal Q}^I $ in the first--order case (see Remark \ref{remPhaseBundlefibrereOnQI} above),
the space $\widetilde{\widetilde{\cal P}}^I$ is now a bundle over
\begin{equation}\label{eqFlaseConfBund}
  \widetilde{\cal Q}^I = J^1\Q = J^1(J^{k-1} \Phi )\ .
\end{equation}
However,  $ \widetilde{\cal Q}^I $ is   a  ``false infinitesimal configuration bundle'': the ``correct infinitesimal configuration bundle''
\begin{equation}\label{eqCorrectConfBund}
 {\cal Q}^I:= J^k \Phi
\end{equation}
is its proper submanifold.
\subsection{Constraining to the ``infinitesimal configuration bundle''}\label{StepEXTRAHiOrd}
The inclusion
\begin{equation}\label{eqIncluItJetSpaces}
 {\cal Q}^I\subset   \widetilde{\cal Q}^I
\end{equation}
of \eqref{eqCorrectConfBund} into  \eqref{eqFlaseConfBund} corresponds to the  following ``holonomic constraints'' imposed on configurations:
\begin{eqnarray}
  \varphi_{ , \lambda} &=& \varphi_\lambda \ , \label{pocz}\\
  \varphi_{\mu ,\lambda} &=& \varphi_{\mu \lambda} \ , \nonumber \\
  \dots &=& \dots \ , \nonumber \\
  \varphi_{{\mu_1} \dots {\mu_{k-1}} , \lambda} &=&
  \varphi_{{\mu_1} \dots {\mu_{k-1}}  \lambda}  \ . \label{kon}
\end{eqnarray}
Submanifolds \eqref{eqIncluItJetSpaces} belong to a large class of canonical submanifolds in iterated jet spaces,   investigated by one of us (GM) in a recent paper \cite{MorenoCauchy}, were the above \emph{formulae} are obtained in the multi--index   notation, viz.
\begin{equation}\label{eqContrMI}
   \varphi_{\mmu ,\lambda} = \varphi_{\mmu \lambda} \ .
\end{equation}
It should be stressed  that the set of identities \eqref{eqContrMI} contains not only relations between the coordinates of $ \widetilde{\cal Q}^I$ but also, for
  $|\mmu|=k$, the  \emph{definition} of the new variable $\varphi_{\mmu \lambda}$, which did not exist before.\par
Now the correct analogues of the forms \eqref{thetaITilde} and \eqref{omegaITilde} can be obtained by restricting the bundles $\widetilde{\widetilde{\cal P}}^I$ and $\widetilde{\widetilde{\cal S}}^I$ over the submanifold ${\cal Q}^I$  of its base manifold $ \widetilde{\cal Q}^I$. The so--obtained bundle over  ${\cal Q}^I$ is denoted by $\widetilde{\cal P}^I$ and it   is equipped with the two canonical forms
\begin{eqnarray}
  \widetilde{\theta}^I
   &=& \left\{ p^\lambda_{\ ,\lambda} \delta \varphi
  + \left( p^\mu  +
  p^{\mu \lambda }_{\ \ ,\lambda}\right) \delta \varphi_\mu +
  \left (p^{ {\mu_1}{\mu_2}} + p^{{\mu_1}{\mu_2} \lambda }_{\ \ \ \ \ \ ,\lambda}\right)\delta \varphi_{{\mu_1}{\mu_2}}
  \right. + \cdots\nonumber \\
   &+&
  \left.
    p^{{\mu_1} \dots {\mu_{k}}} \delta \varphi_{{\mu_1} \dots {\mu_{k}} } \right\}
  \otimes
     \mbox{\rm d}^n x   \, ,\label{eqThetatileIHI}
\end{eqnarray}
and
\begin{eqnarray}
  \widetilde{\omega}^I
   &=& \left\{ \delta p^\lambda_{\ ,\lambda} \wedge \delta \varphi
  + \delta \left( p^\mu  +
  p^{\mu \lambda }_{\ \ ,\lambda}\right) \wedge \delta \varphi_\mu +
  \delta \left (p^{ {\mu_1}{\mu_2}} + p^{{\mu_1}{\mu_2} \lambda }_{\ \ \ \ \ \ ,\lambda}\right)\wedge \delta \varphi_{{\mu_1}{\mu_2}}
  \right.+ \cdots\nonumber\\
   &+&
  \left.
    \delta p^{{\mu_1} \dots {\mu_{k}}} \wedge \delta \varphi_{{\mu_1} \dots {\mu_{k}} } \right\}
  \otimes
     \mbox{\rm d}^n x   \ ,\label{eqOmegatileIHI}
\end{eqnarray}
obtained by restriction from \eqref{eqThetaTildeTilde} and \eqref{eqOmegaTildeTilde}, respectively.
\subsection{The infinitesimal phase bundle}\label{Step3HiOrd}
After this preliminary constraining of the base manifold $ \widetilde{\cal Q}^I$, we proceed with the symplectic reduction of the bundle $\widetilde{\cal P}^I$,  along the same lines sketched in Section \ref{ToyStep3}. Namely, the same formula \eqref{defInfPhasBund}, rewritten below, defines now the leaf  space
\begin{equation}\label{defInfPhasBundHI}
    {\cal P}^I :=\frac{ \widetilde{\cal P}^I }{{\textrm{degeneracy of }}\, \widetilde{\omega}^I} \ 
 \end{equation}
of the space
 $\widetilde{\cal P}^I$ constructed in the above Section \ref{StepEXTRAHiOrd}, with respect to  the degeneracy distribution
of the 2--form  $\widetilde{\omega}^I$ defined by \eqref{eqOmegatileIHI}. \par

\begin{definition}
 The bundle ${\cal P}^I$ defined by \eqref{defInfPhasBundHI} is called the \emph{infinitesimal phase bundle}. It is equipped with the (volume--form--valued) \emph{infinitesimal symplectic form}  $\omega^I$, defined as the reduction of $\widetilde{\omega}^I$.
\end{definition}

Because jet coefficients $\varphi_{{\mu_1} \dots {\mu_k}}$ are totally symmetric, formula \eqref{eqOmegatileIHI} proves that a leaf of the degeneracy distribution, i.e., a point of ${\cal P}^I$, is uniquely determined by the following parameters:
\begin{eqnarray}
  j &=& p^\lambda_{\ ,\lambda} \ ,\label{j-1}\\
  j^\mu  &=&  p^\mu  +
  p^{\mu \lambda }_{\ \ ,\lambda} \ , \nonumber\\
  j^{ {\mu_1}{\mu_2}} &=& p^{ ({\mu_1}{\mu_2})} + p^{{\mu_1}{\mu_2} \lambda }_{\ \ \ \ \ \ ,\lambda} \ , \nonumber\\
  \dots &=& \dots \ , \nonumber\\
  j^{{\mu_1} \dots {\mu_{k}}} &=& p^{({\mu_1} \dots {\mu_{k}})} \ , \label{j-k}
\end{eqnarray}
where the bracket denotes the complete symmetrisation.\footnote{Since   \eqref{sym-1}--\eqref{sym-2}, we began adopting the physicists' notations for the complete symmetrisation.} Being defined by   \eqref{j-1}--\eqref{j-k} as momenta canonically conjugate to the jet coefficients $\varphi_{{\mu_1} \dots {\mu_{l}} }$, the currents  $j^{{\mu_1} \dots {\mu_{l}}}$, $l=1,2,\dots , k$, are totally symmetric {\em a priori}. Consequently, as a result of  this symplectic reduction, only the completely symmetric part $p^{({\mu_1} \dots {\mu_{l}})}$ of the momenta $p^{{\mu_1} \dots {\mu_{l}}}$ come into play. We stress that at the beginning of our construction no symmetry was imposed on the last index: see \eqref{sym-1}. Splitting the momenta into their totally symmetric part $s$ and the remaining part $r$, namely
\begin{equation}\label{split-p}
    p^{{\mu_1} \dots {\mu_{l}}} = s^{{\mu_1} \dots {\mu_{l}}} + r^{{\mu_1} \dots {\mu_{l}}} \ ,
\end{equation}
where $s^{{\mu_1} \dots {\mu_{l}}} := p^{({\mu_1} \dots {\mu_{l}})}$ and $r^{({\mu_1} \dots {\mu_{l}})} = 0$, we see that the non--symmetric part $r^{{\mu_1} \dots {\mu_{l}}}$ corresponds to the  irrelevant or ``gauge'' degrees of freedom and disappear when we pass to the quotient ${\cal P}^I$.

\begin{theorem}\label{PI=SI}
The infinitesimal phase bundle can be obtained equivalently {\em via} the symmetric version of the theory:
\begin{equation}\label{defInfPhasBundHI-s}
    {\cal P}^I \simeq {\cal S}^I :=\frac{ \widetilde{\cal S}^I }{{\textrm{degeneracy of }}\, \widetilde{\omega}^I}  \ .
\end{equation}
\end{theorem}

In the symmetric version of the theory the non--symmetric part $r^{{\mu_1} \dots {\mu_{l}}}$ of the momentum drops out from the very beginning and the symmetrisation operator in the definition \eqref{j-1}--\eqref{j-k} of the currents $j$ may be skipped.

The forms \eqref{eqThetatileIHI} and \eqref{eqOmegatileIHI} can be restricted to   the quotient space ${\cal P}^I$. This way,  we obtain the higher--order version of \eqref{thetaISenzaTilde} and
\eqref{OmegaIDifferenzialeDiThetaI}, respectively. In particular, using the  multi--index notation, we obtain the following coordinate expression for the infinitesimal symplectic form:
\begin{eqnarray}
   {\omega}^I
   &=& \sum_{|\mmu|\leq k} \delta j^{\mmu } \wedge
    \delta \varphi_{\mmu}
  \otimes
      \mbox{\rm d}^n x   \ ,\label{eqDefOmegaIHI}
\end{eqnarray}
where, for $1 \leq |\mmu| \leq (k-1)$,
\begin{eqnarray*}
  j &=& p^\lambda_{\ ,\lambda}
  \ \ \ \ \mbox{\rm for} \ \ \  |\mmu| = 0 \ ,\\
  j^{\mmu } &=& p^{\mmu } + p^{\mmu \lambda}_{\ \ ,\lambda}
  \ \ \ \ \mbox{\rm for} \ \ \ 1 \leq |\mmu| \leq (k-1)\ ,\\
  j^{\mmu } &=& p^{\mmu }
  \ \ \ \ \mbox{\rm for} \ \ \  |\mmu| = k\ .
\end{eqnarray*}
Also the 1--form $\widetilde{\theta}^I$ is compatible with this reduction and defines on ${\cal P}^I$ a primitive form
\begin{eqnarray}\label{theta^I}
  {\theta}^I
   &=& \left( j \delta \varphi
  + j^\mu   \delta \varphi_\mu +
  j^{ {\mu_1}{\mu_2}} \delta \varphi_{{\mu_1}{\mu_2}}
  + \cdots +
    j^{{\mu_1} \dots {\mu_{k}}} \delta \varphi_{{\mu_1} \dots {\mu_{k}} } \right)
  \otimes
     \mbox{\rm d}^n x   \\
    &=&
    \sum_{|\mmu|\leq k}   j^{\mmu }
    \delta \varphi_{\mmu}
  \otimes
      \mbox{\rm d}^n x \  \nonumber
\end{eqnarray}
for ${\omega}^I$, in the sense that $\delta \theta^I = \omega^I$.

\begin{theorem}\label{ThToyModelHI}
 There is a canonical identification
 \begin{equation}\label{eqPrimaIdentificazioneCanonicaHI}
{\cal P}^I  \cong  V^*{\cal Q}^I
    \otimes_{{\cal Q}^I} \Lambda^{n}
    M
\end{equation}
of linear bundles over ${\cal Q}^I$.
\end{theorem}
\begin{proof}
Carried out in Section \ref{SecByGiovanni}.
\end{proof}
Observe that the identification \eqref{eqPrimaIdentificazioneCanonicaHI} looks exactly the same as the similar identification
\eqref{eqPrimaIdentificazioneCanonica} from Theorem \ref{ThToyModel}: the difference is hidden in the   definition of the ``infinitesimal configuration bundle'' (compare \eqref{eqInfConfBund} and \eqref{eqCorrectConfBund}).
As before, the canonical identification \eqref{eqPrimaIdentificazioneCanonicaHI} corresponds to a   family of identifications of bundles over ${\cal Q}^I_x$:
\[
    {\cal P}^I_x \cong  T^*{\cal Q}^I_x
    \otimes \bigwedge^{n}
    T^*_x M \ , \quad \forall x\in M\ .
\]

\subsection{The higher--order Euler--Lagrange equations as a Lagrangian submanifold}\label{subHOELeqs}
Now we can carry out the last step and, in analogy with Section \ref{subELdynamics}, write down the Euler--Lagrange equations for a \emph{higher--order} Lagrangian as the generating \emph{formula} for a Lagrangian submanifold in $\P^I$. The same symbol ${\cal L}$, used in Section \ref{subELdynamics} for a first--order Lagrangian density, corresponds now to   ${\cal L} = L \, \mbox{\rm d}^n x$, where
\begin{equation}
    L = L(\varphi, \varphi_{\mu} , \varphi_{{\mu_1}{\mu_2}}, \dots\label{eqHILagr}
    , \varphi_{{\mu_1} \dots {\mu_{k}}})\, .
\end{equation}
Corollary \ref{corELeqsAreLagrSubs} is repeated verbatim here, except for the increased lengths of the system of first--order PDEs.
\begin{corollary}\label{corELeqsAreLagrSubsHI}
The Euler--Lagrange equations determined by ${\cal L}$ are equivalent to the generating formula for the Lagrangian submanifold ${\cal D} \subset {\cal P}^I$, according to equation
(cf.~also \eqref{eqGenFormula}):
\begin{equation}\label{defDynamicsHI}
\delta {\cal L} = \left. \theta^I\right|_{\cal D} \ .
\end{equation}

\end{corollary}
\begin{proof}
Comparing with \eqref{theta^I} and taking into account the definition of the currents $j$, we see that  the equation \eqref{defDynamicsHI} captures the following list of first--order PDEs
\begin{eqnarray*}
   p^\lambda_{\ ,\lambda} &=& \frac{\partial L}{\partial \varphi}\, , \\
   p^\mu  +
  p^{\mu \lambda }_{\ \ ,\lambda} &=& \frac{\partial L}{\partial \varphi_{\mu}}\, ,\\
  p^{( {\mu_1}{\mu_2})} + p^{{\mu_1}{\mu_2} \lambda }_{\ \ \ \ \ \ ,\lambda}
  &=& \frac{\partial L}{\partial \varphi_{{\mu_1}{\mu_2}}}\, ,\\
  \dots &=& \dots \ , \\
  p^{({\mu_1} \dots {\mu_{k}})}
  &=& \frac{\partial L}{\partial \varphi_{{\mu_1} \dots {\mu_{k}}}}\, ,
\end{eqnarray*}
which, in an equivalent form, read
\begin{eqnarray}
  p^{({\mu_1} \dots {\mu_{k}})}
  &=& \frac{\partial L}{\partial \varphi_{{\mu_1} \dots {\mu_{k}}}}
  - 0\, ,\label{pierwsze}\\
  p^{({\mu_1}\dots{\mu_{k-1}})}
  &=& \frac{\partial L}{\partial \varphi_{{\mu_1}\dots {\mu_{k-1}}}} -
  \partial_\lambda p^{{\mu_1}\dots {\mu_{k-1}} \lambda }\, ,\nonumber \\
  \dots &=& \dots \ , \nonumber\\
  p^{ ({\mu_1}{\mu_2})}
  &=& \frac{\partial L}{\partial \varphi_{{\mu_1}{\mu_2}}}
  - \partial_\lambda p^{{\mu_1}{\mu_2}\lambda }\, ,\nonumber\\
  p^\mu
   &=& \frac{\partial L}{\partial \varphi_{\mu}}
   -\partial_\lambda p^{\mu \lambda }\, ,\label{ostatnje}
    \\
    0 &=& \frac{\partial L}{\partial \varphi}
    - \partial_\lambda p^{\lambda}\, .\label{eqELHI}
\end{eqnarray}
Finally, observe that the Euler--Lagrange equations determined by ${\cal L}$ appear in   \eqref{eqELHI}, whereas the remaining equations \eqref{pierwsze}--\eqref{ostatnje} contain the definition of the canonical momenta, in both the symmetric ${\cal S}^I$ and the non-symmetric ${\cal P}^I$ versions of the higher--order theory.  \end{proof}

\subsection{Momentum gauge and how to remove it}\label{momentum-gauge}

This is the appropriate moment to clarify  how, in the non--symmetric version of the theory,    the  definition \eqref{pierwsze}--\eqref{ostatnje} of the momenta depend  upon the gauge degrees of freedom. The main difficulty with respect to the $1\St$ order case is that any gauge adjustment in one of equations \eqref{pierwsze}--\eqref{ostatnje} propagates through the whole sequence.\par
  First, observe that the  derivatives of the Lagrangian function  $L$, with respect to the jet variables $\varphi_{{\mu_1} \dots {\mu_{l}}}$, are   unambiguously defined as totally symmetric tensor densities. Hence, equations \eqref{pierwsze}--\eqref{ostatnje}
  determine  only the symmetric part $s^{{\mu_1} \dots {\mu_{l}}}$ in the decomposition \eqref{split-p} of the momenta $p^{{\mu_1} \dots {\mu_{l}}}$. So, at a first glance, the remaining part $r^{{\mu_1} \dots {\mu_{l}}}$ of the momenta is totally free. This is not entirely true. Indeed,  derivatives $\partial_\lambda  r^{{\mu_1} \dots {\mu_{l-1}} \lambda}$ enter   the right hand side of \eqref{pierwsze}--\eqref{ostatnje} and force us to modify the lower--order symmetric part, namely $p^{({\mu_1} \dots {\mu_{l-1}})}$, in such a way that  a modification of $r^{{\mu_1} \dots {\mu_{l}}}$ by a term $\chi^{{\mu_1} \dots {\mu_{l}}}$, where $\chi^{({\mu_1} \dots {\mu_{l}})}=0$, implies the next modification, i.e.,  $p^{({\mu_1} \dots {\mu_{l-2}})}$ has to be modified ({\em modulo} a possible change of sign) by $\partial_{\lambda_1}\partial_{\lambda_2}  \chi^{{\mu_1} \dots {\mu_{l-2}} {\lambda_1}{\lambda_2}}$.\par
  It is easy to convince oneself  that a modification of a single non--symmetric object $r^{{\mu_1} \dots {\mu_{l}}}$ triggers a chain of modifications in the lower--order symmetric objects by means of the iterated   derivatives
\[
    \partial_{\lambda_1}\cdots \partial_{\lambda_m}  \chi^{{\mu_1} \dots {\mu_{l-m}} {\lambda_m}\dots {\lambda_1}} \ ,
\]
which eventually affect also equation \eqref{eqELHI}, by means of the $l\Th$ order derivative
\[
    \partial_{\lambda_1}\cdots \partial_{\lambda_l}  \chi^{{\lambda_l}\dots {\lambda_1}} =
    \partial_{\lambda_1}\cdots \partial_{\lambda_l}  \chi^{({\lambda_l}\dots {\lambda_1})}= 0\ .
\]
In other words,   given a solution of the field equation \eqref{eqELHI}, the above modifications produce another solution, but these are physically equivalent as they define   the same section of the configuration bundle ${\cal Q}$. Moreover, among all the equivalent solutions there is one with totally symmetric momenta $p^{{\mu_1} \dots {\mu_{l}}} = s^{{\mu_1} \dots {\mu_{l}}}$, since the above modification procedure can be used to annihilate the non--symmetric part $r^{{\mu_1} \dots {\mu_{l}}}$.\par

If we want to keep the one--to--one correspondence between the sections of the bundle $\Phi$ which satisfy the $2k\Th$  order system of Euler--Lagrange equations, and their canonical representation, i.e.,~the sections of the momentum bundle ${\cal P}$ which satisfy the system of first--order equations \eqref{pocz}--\eqref{kon} and \eqref{pierwsze}--\eqref{eqELHI}, we must  restrict \emph{from the very beginning}  the phase bundle ${\cal P}$  to its sub--bundle ${\cal S}\subset {\cal P}$, composed of totally symmetric momenta, avoiding the redundancy carried by the non--symmetric momenta $r^{{\mu_1} \dots {\mu_{l}}}$. This restriction of the phase bundle corresponds to the observation that later on, the momentum is always applied to {\em holonomic} jets only and, consequently, its non--symmetric part can be skipped form the very beginning. We conclude that the symmetric version of the theory, based on the symmetric phase bundle ${\cal S}$, is gauge--free, which  is very appealing from the conceptual point of view.  More details about the construction of ${\cal S}$     are put off in the Appendix \ref{app-momentum-gauge}.\par
On the other hand, keeping the gauge degrees of freedom represented by the non--symmetric part of the momenta is sometimes useful from the computational point of view, as illustrated by   Section \ref{secModLagrTotDiv} below. We stress, however, that both approaches are perfectly equivalent because the non--symmetric part of the momenta never comes into play:
every solution of the field equations on ${\cal P}$ has a unique, equivalent representation as a section of ${\cal S}$, fulfilling the corresponding symmetric version of the field equations. Indeed,   the infinitesimal phase space ${\cal P}^I$, obtained either from $J^1 {\cal P}$ or from $J^1 {\cal S}$ {\em via} the symplectic reduction is the same, so that both versions of the theory are equivalent.\par
It is worth noticing that the formula \eqref{eqELHI} extends immediately to infinite jets (see, e.g., \cite{MR2647288}).

\subsection{Modifying the Lagrangian by a total divergence}\label{secModLagrTotDiv}

Supplementing a Lagrangian by a total divergence does not influence the Euler--Lagrange equations, since the  new terms arising in the corresponding action functional are only boundary ones. A decent ``canonical version'' of the theory must follow this ``mathematical folklore''. However, adding, e.g.,~100 new derivatives to the Lagrangian, produces {\em a priori} 100 new momenta. How do we understand the equivalence? In this section we show that, indeed, the original theory and the theory based on the new, artificially obtained ``higher--order Lagrangian'', are equivalent in the sense of the symplectic reduction. \par

To begin with, take a vector--density--valued function
\begin{equation}\label{F-on-jets}
    F:=F^\lambda \dens{\lambda}
\end{equation}
defined on the bundle $J^{l-1}\Phi$, i.e.,
\begin{equation}
F^\lambda=F^\lambda(x,\varphi, \varphi_{\mu} , \varphi_{{\mu_1}{\mu_2}}, \dots
    , \varphi_{{\mu_1} \dots {\mu_{l-1}}})\ ,\quad \lambda=1,2,\ldots, n\ ,
\end{equation}
and consider the  $l\Th$  order Lagrangian $ {\cal L}_0$ defined as the   divergence of \eqref{F-on-jets}:
\begin{equation}\label{dF=L}
    {\cal L}_0 = {\rm d}F = (\partial_\lambda F^\lambda ) \mbox{\rm d}^n x = L_0 \mbox{\rm d}^n x \ .
\end{equation}
Below, we prove that  the dynamics corresponding to  $ {\cal L}_0$, trivial from the variational point of view,  is also symplectically trivial.
\begin{theorem}\label{trivial}
To  {any} section $ M \ni x \stackrel{f}{\rightarrow} \varphi(x) \in \Phi_x$ of the bundle $\Phi$ corresponds biuniquely a section
\begin{equation}\label{trivial-P}
    \sigma_f:= \left( j^{l-1}(f), p^\mu\right)
\end{equation}
 of $\P$ which is a solution of the Euler--Lagrange equations \eqref{pierwsze}--\eqref{eqELHI}. The   momenta $p^{\overline{\mu}\lambda}$, $0 \le |\overline{\mu}| \le l-1$,  appearing in \eqref{trivial-P} are unambiguously determined by the $l-1\St$ jet of $F$ \emph{via}\begin{equation}\label{trivial-P-P}
    p^{\overline{\mu}\lambda} := \frac{\partial F^\lambda}{\partial \varphi_{\overline{\mu}}} \ .
\end{equation}
\end{theorem}
\begin{proof}
As a consequence of \eqref{dF=L}, we have
\begin{equation}\label{dF=L=cons}
    L_0 = \frac{\partial F^\lambda}{\partial x^\lambda} + \sum_{0\le |{\overline{\nu}}|\le l-1}
    \frac{\partial  F^\lambda}{\partial \varphi_{\overline{\nu}}}  \varphi_{\overline{\nu}\lambda}
\end{equation}
and, consequently,
\begin{equation}\label{dF=L=cons-1}
    \frac{\partial L_0}{\partial \varphi_{\overline{\mu}}}  = \frac{\partial^2 F^\lambda}{\partial \varphi_{\overline{\mu}} \partial x^\lambda} + \sum_{0\le |{\overline{\nu}}|\le l-1}
    \frac{\partial^2 F^\lambda}{\partial \varphi_{\overline{\mu}}\partial \varphi_{\overline{\nu}}}  \varphi_{\overline{\nu}\lambda} + p^{({\overline{\mu}})} \ ,
\end{equation}
the last term coming from the linear, explicit dependence of \eqref{dF=L=cons} upon $\varphi_{\overline{\mu}}=\varphi_{\overline{\nu}\lambda} $ for $|{\overline{\nu}}|= l$, and definition \eqref{trivial-P-P} (of course, $p^{\overline{\mu}} = 0$ for $|{\overline{\mu}}| = 0$). But, according to \eqref{trivial-P-P}, we have
\begin{equation}\label{divergence-p}
    \partial_\lambda p^{\overline{\mu}\lambda} =
    \frac{\partial^2 F^\lambda}{\partial x^\lambda\partial \varphi_{\overline{\mu}} } + \sum_{0\le |{\overline{\nu}}|\le l-1}
    \frac{\partial^2 F^\lambda}{\partial \varphi_{\overline{\nu}}\partial \varphi_{\overline{\mu}}}  \varphi_{\overline{\nu}\lambda} \
\end{equation}
and, therefore, the Euler--Lagrange equations \eqref{pierwsze}--\eqref{eqELHI} are automatically satisfied:
\begin{equation}\label{E-L-final}
    \frac{\partial L_0}{\partial \varphi_{\overline{\mu}}} - \partial_\lambda p^{\overline{\mu}\lambda} = p^{({\overline{\mu}})} \ .
\end{equation}
\par \

\end{proof}
Theorem \ref{trivial} can be made  ``totally symmetric'' according   to the philosophy discussed in Section \ref{momentum-gauge}, by assigning to the section $f$ of $\Phi$ a   section $s_f$ of the sub--bundle ${\cal S} \subset {\cal P}$, instead of a section $\sigma_f$ of the whole bundle $\P$. In this version, however,  formula \eqref{trivial-P-P}  becomes much more complicated and reduces to $s^{\overline{\mu}}= p^{(\overline{\mu})}$ only for the highest--order momenta, i.e.,~when $|{\overline{\mu}}|= l$, whereas  lower--order momenta need further modifications, according to the analysis carried out in Section \ref{momentum-gauge}. This is a typical circumstance when   the \virg{totally  symmetric version} of the theory turns out to be  computationally  unfriendly. The authors guess that this is the {\em very reason} why a consistent ``canonical'' theory for higher order Lagrangians was never written before, even if all the ingredients were ready more than 35 years ago\ldots .

\begin{remark}
 Theorem \ref{trivial}  can be reformulated as follows. Treat formulae \eqref{trivial-P}--\eqref{trivial-P-P}   as a definition of  the momentum constraints in the infinitesimal configuration bundle ${\cal P}^I$. Such  constraints define a Lagrangian submanifold, restricted to which  the infinitesimal symplectic form becomes totally degenerate. In other words, there is a  unique degeneracy leaf, so that  the corresponding symplectic reduction leads to a trivial space, i.e.,     trivial dynamics.
\end{remark}
Consider now a non--trivial Lagrangian ${\cal L}$ of order $k$ and supplement it by the complete divergence \eqref{dF=L}:
\begin{equation}\label{L+divergence}
\widetilde{ {\cal L} }:={\cal L} + {\rm d}F\ .
\end{equation}
Suppose first that $l \le k$. The following theorem is a simple corollary of our previous considerations:
\begin{theorem}\label{trivial-1}
Consider the symplectomorphism ${\cal F}:{\cal P}^I \mapsto {\cal P}^I $ of ${\cal P}^I$ generated by the shift of momenta:
\begin{equation}\label{momentum-shift}
    p^{\overline{\mu}\lambda} \rightarrow p^{\overline{\mu}\lambda} + \frac{\partial F^\lambda}{\partial \varphi_{\overline{\mu}}} \ .
\end{equation}
If $\sigma$ is a section of ${\cal P}$ satisfying the Euler--Lagrange equations \eqref{pierwsze}--\eqref{eqELHI} generated by ${\cal L}$, then ${\cal F}\circ \sigma$ satisfies the Euler--Lagrange equations generated by $\widetilde{ {\cal L} }$ and {\em vice--versa}.

\end{theorem}

If $l \ge k$ then, {\em a priori}, \eqref{L+divergence} increases the order of the variational problem in question. Nevertheless, due to   Theorem \ref{trivial-1}, the theory is equivalent to the $k\Th$ order theory and the equivalence is given by the inverse of the shift \eqref{momentum-shift}. In the infinitesimal phase space ${\cal P}^I$ of this theory we have, after such a shift, the momentum constraints $p^{(\overline{\mu})}=0$ for $|\overline{\mu}| > k $, implied by the higher--order Euler--Lagrange equations. The symplectic reduction with respect to these constraints reproduces the phase bundle and the dynamics corresponding to ${\cal L}$. The ``symmetric'' version of the formula \eqref{momentum-shift} follows immediately {\em via} the chain of gauge transformations discussed in Section \ref{momentum-gauge} and is much more complicated. This is why we decided to give the parallel construction of  both the symmetric and the non--symmetric versions of the theory.

\section{Representation of the infinitesimal phase bundle as the bundle of vertical covectors on the infinitesimal configuration bundle}\label{SecByGiovanni}

In the key Theorems \ref{ThToyModel} and \ref{ThToyModelHI} above we made use of the fact that the infinitesimal phase bundle ${\cal P}^I$ can be thought of as the bundle of (volume--form--valued) vertical covectors on the infinitesimal configuration bundle ${\cal Q}^I$, in order to be able to claim, in the subsequent Corollaries \ref{corELeqsAreLagrSubs} and \ref{corELeqsAreLagrSubsHI}, that the vertical differential of a Lagrangian density is, in fact, a Lagrangian submanifold of ${\cal P}^I$. In this section we clarify this crucial property, namely,  we show that there is a canonical bundle identification
\begin{equation}\label{eqDaDimostrare}
\xymatrix{
\P^I \ar@{=}[rr]\ar[dr]&& V^\ast  \Q^I \otimes_{ \Q^I}   \Lambda^{n}M\ar[dl]\\
& M\, .&
}
\end{equation}
Recall that, by its very definition
\eqref{defInfPhasBund}, $\P^I$ is the leaf space of the degeneracy distribution of $\widetilde{\omega}^I$, so that working directly on it may be a little uncomfortable. So, we shall adopt an indirect  strategy, and obtain the desired result   \eqref{eqDaDimostrare}    as an immediate consequence of    Theorem \ref{thEgr} below.
\subsection{The unconstrained case}
\begin{theorem}\label{thEgr}
 There is a canonical mapping $\Psi$ respecting the fibrations over $\Q^I$, i.e., making the following diagram commutative
\begin{equation}\label{eqDaDimostrareMODIF}
\xymatrix{
J^1{\P}\ar@{->>}[rr]^{\Psi}\ar[dr]&&  V^\ast  \Q^I \otimes_{ \Q^I}   \Lambda^{n}M\ar[dl]\\
&\Q^I,&
}
\end{equation}
and such that, for every point $\eta$ of the target space, the inverse image $\Psi^{-1}(\eta)$ is a degeneracy leaf of $\widetilde{\omega}^I$. In particular, $\Psi$ is surjective.
\end{theorem}
%
Before commencing the proof, a key preliminary result must be given. Indeed, by its definition \eqref{eqDefPhaseBundle}, $\P$ is made of (vector--density--valued) vertical covectors on $\Q$, so that, informally speaking,
\begin{equation}
J^1{\P}=J^1(V^\ast Q\otimes\Lambda^{n-1}M)\, .\label{eqThEgr1}
\end{equation}
On the other hand, in view of the definition \eqref{eqInfConfBund} of the infinitesimal configuration bundle $\Q^I$, the right--hand side of \eqref{eqDaDimostrareMODIF} is made of (volume--form--valued) vertical covectors on  $\Q^I$, viz.
\begin{equation}
V^\ast  \Q^I \otimes   \Lambda^{n}M=V^\ast  (J^1\Q) \otimes  \Lambda^{n}M\, .\label{eqThEgr2}
\end{equation}
Forgetting about the base spaces involved, which were deliberately skipped in the tensor products, a quick comparison of the right--hand sides of \eqref{eqThEgr1} and \eqref{eqThEgr2} reveals  that they are made of the same  symbols $J^1$, $V$, $\Q$, $\Lambda$ and $M$, with only two, yet remarkable, differences:
\begin{itemize}
\item the order of \virg{$J^1$} and \virg{$V$} is interchanged;
\item \virg{$\Lambda$} stands for $(n-1)$--forms in  \eqref{eqThEgr1} and  for $n$--forms in \eqref{eqThEgr2}.
\end{itemize}
So, we should expect that  the desired mapping \eqref{eqDaDimostrareMODIF} stems from a natural isomorphism
\begin{equation}\label{eqJ1VugualeVJ1}
J^1V\Q\cong V J^1\Q\, ,
\end{equation}
followed by a differentiation of  $(n-1)$--forms.
\begin{proposition}\label{ProposizioneFondamentale}
 The identification \eqref{eqJ1VugualeVJ1} is valid and reads
 \begin{equation}
(\phi,v,\phi_\mu,v_\mu)\longleftrightarrow (\phi,\phi_\mu,v,v_\mu)\label{eqJ1VugualeVJ1COORD}
\end{equation}
 in local coordinates.
\end{proposition}
\begin{proof}
 Local identity \eqref{eqJ1VugualeVJ1COORD} should be enough to convince oneself of the validity of \eqref{eqJ1VugualeVJ1}. A general rigorous proof takes much more space and it is put off (see Section \ref{appVJ1UgJ1V} later on).
\end{proof}
We are now in position to prove the main result of this section. The   idea of the proof is rather simple, and it comes down to  using  Proposition \ref{ProposizioneFondamentale} to make an element of $J^1{\cal P}$ act on  vertical covectors on ${\cal Q}^I$, and then   differentiating the result in order to get an $n$--form. The only difficulty consists in keeping track of the correct bundle structure one must work with. In particular, we shall make use of the projection $\Pi:\P\to Q$ (cf.     diagram \eqref{eqDaDimostrareMODIF1}).
\begin{proof}[Proof of Theorem \ref{thEgr}] It will be carried out in three steps. First, we define the map $\Psi$, second we prove that the definition is well--behaved and, last, we show that the fibres of  $\Psi$ are precisely the degeneracy leaves of $\widetilde{\omega}^I$.\par
 In order to define the image $\Psi(\zeta) \in V^\ast  \Q^I \otimes_{ \Q^I }   \Lambda^{n}M $ of the generic  element $\zeta\in  J^1{\P}$, observe initially  that, since $\zeta$ is attached  to a point, let us call it  $x_0$, of $M$, i.e.,~$\zeta \in J^1_{x_0}{\P}$, then its image  $\Psi(\zeta)$ has to be sought for  in the fibre $ V^\ast  \Q^I_{x_0} \otimes_{ \Q^I_{x_0}}   \Lambda_{x_0}^{n}M $. To this end, $\zeta$ must act on a generic  vector $w$, which is tangent to the bundle $\Q^I$, and is also vertical with respect to the projection down to $M$. Moreover, since $\Psi$ has to be a morphism of bundles over $\Q^I$, the vector $w$ must, in particular, belong to $T_{q^I}\Q^I_{x_0}$, where $q^I \in \Q^I $ is the left--hand side projection of $\zeta$ in diagram \eqref{eqDaDimostrareMODIF}.\par
 In other words, we have to look for a natural  $(\Lambda^{n}_{x_0}M )$--valued pairing
\begin{equation}\label{inf-evaluation}
    <w , \Psi(\zeta) > \ \in \Lambda^{n}_{x_0}M \, .
\end{equation}
 In order to define correctly  \eqref{inf-evaluation}, it is indispensable  use Proposition \ref{ProposizioneFondamentale}. Indeed, in view of
\begin{equation}\label{wertykal}
    w \in V \Q^I = V J^1\Q = J^1 (V\Q) \, ,
\end{equation}
both $w$ and $\zeta$ in \eqref{inf-evaluation} can be regarded as the first jets of dual quantities, i.e., apt to be paired each other. More precisely, we can ``extend'' both $w$ and $\zeta$ to the first--order jet of a section $v$ and $p$ of the bundles $\P$ and $V\Q$ over $M$, respectively. On the top of that, $v$ and $p$ can be chosen lying over the \emph{same} section $s$ of $\Q$. In practice, we have   chosen a section $p$  which represents $\zeta$, i.e.,~such that $\zeta = j^1_{x_0}(p) \in J^1_{x_0}{\P}$, we defined $s$ as its projection, $s:= \Pi(p)$, and, in view of \eqref{wertykal}, we   represented the vector $w$ as the first jet of a section $v$ of the bundle $V\Q$, i.e.,
\begin{equation}\label{w=jv}
    w = j^1_{x_0} v \, ,
\end{equation}
where $v$ has been chosen in such a way\footnote{This means that $v$ is a vertical vector field ``along'' the graph of $s$.}  that its projection on $\Q$ is equal to $s$.    Diagram \eqref{eqDaDimostrareMODIF1-bis} below gives some perspective on the sections introduced so far:     the departing point is the element $\zeta$ and, while its projections $q^I$ and $x_0$ are uniquely defined, the sections $p$ and $s$ representing $\zeta$ and $q^I$, respectively, on the point $x_0$, are arbitrary. Similarly for $w$ and its representing section $v$.
\begin{equation}\label{eqDaDimostrareMODIF1-bis}
\xymatrix{
**[l]\zeta\in J^1{\P}\ar@{->>}[rr]^{\Psi}\ar[dr]&&  V^\ast  \Q^I \otimes_{ \Q^I}   \Lambda^{n}M\ar[dl]\\
&**[l]q^I\in \Q^I \ar[d] &**[r] J^1V\Q\ni w\ar[l]\\
&\Q\ar[d]^\pi& V\Q\ar[l] \\
&**[l]x_0\in M\ar@{.>}@/^3pc/[uu]^{j_1(s)}\ar@{.>}@/^1pc/[u]^{s} \ar@{.>}@/^4pc/[uulu]^{j_1(p)} \ar@{.>}@/_1pc/[ru]_{v}\ar@{.>}@/_5pc/[ruu]_{j_1(v)}&
}
\end{equation}
What really matter now is that, for \emph{every} $x\in M$ we have
\begin{eqnarray}
    p(x) &\in &V^\ast_{s(x)}  \Q_x \otimes_{ \Q_x}   \Lambda^{n-1}_xM\, ,\label{eqPiIcs}\\
    v(x)& \in& V_{s(x)}  \Q_x \, ,\label{eqViIcs}
\end{eqnarray}
i.e., \eqref{eqPiIcs} can be naturally paired with  \eqref{eqViIcs}.
Consequently,
\begin{equation}\label{vector-dens}
\left.<v,p>\right|_x:=    <v(x) , p(x) > \ \in  \Lambda^{n-1}_xM\, , \quad x\in M\, ,
\end{equation}
defines a    vector density on $M$.
Now we are in position  to define the pairing \eqref{inf-evaluation}, i.e., the value of $\Psi(\zeta)$ on $w$, by  taking the divergence (exterior derivative) of \eqref{vector-dens} at the point $x_0$, viz.
\begin{equation}\label{inf-evaluation-d}
    <w , \Psi(\zeta) > := \left. \left( \mbox{\rm d} <v , p > \right) \right|_{x_0} \
    \in \Lambda_{x_0} ^{n}M\, ,
\end{equation}
and the first part of the proof is complete.\par
In order to   check that \eqref{inf-evaluation-d} is well--defined, it is worth recalling that $p$ (resp., $v$) denotes a $\Lambda^{n-1}M$--valued vertical form (resp., a vertical vector field) on $\Q$ \emph{along the graph of $s$}. As such, they can be written as
\begin{equation}\label{eqPiEVuInCoord}
p=p^\mu\delta\varphi\otimes \dens{\mu}\, , \quad v=v\frac{\partial}{\partial \varphi}\, ,
\end{equation}
respectively,\footnote{Recall that we dropped the upper index $K$ from $\varphi^K$; had we not, the coefficient $v$ in \eqref{eqPiEVuInCoord} above would keep a lower index $K$, thus distinguishing it from the vector $v$.} bearing in mind that, by evaluating at $x\in M$, one gets
\begin{equation}
p(x)=p^\mu(x)\left.\delta\varphi\right|_{s(x)}\otimes \left.\dens{\mu}\right|_x\, , \quad v(x)=v(x)\left.\frac{\partial}{\partial \varphi}\right|_{s(x)}\, .\label{eqPiEVuInCoord2}
\end{equation}
Observe that in \eqref{eqPiEVuInCoord} there is no evidence  of the section $s$, which appears only in the evaluation   \eqref{eqPiEVuInCoord2} of \eqref{eqPiEVuInCoord}  at $x$. Nevertheless, since $\delta\varphi$ and $\frac{\partial}{\partial \varphi}$ are dual each other, once they are paired, there will be no trace left of the section $s$. So, by pairing     $p$   with $v$, one gets a genuine  $(n-1)$--form on $M$, i.e.,   independent on $s$ and, as such, it can be differentiated:
\begin{equation}\label{inf-evaluation-d-Coord}
d\langle p,v\rangle=(v\partial_\mu p^\mu+p^\mu \partial_\mu v)\dd^n x\, .
\end{equation}
Last formula \eqref{inf-evaluation-d-Coord} provides us with a   description of the right--hand side of \eqref{inf-evaluation-d} in terms of the coordinates of $v$ and $p$ given in \eqref{eqPiEVuInCoord}. Recall now that  the first jet $j_1(v)$ of $v$ (resp., $j_1(p)$ of $p$) is uniquely determined by the functions $(v,v_\mu)$, where $v_\mu=\partial_\mu v$ (resp., $(p^\mu,p^\mu_\nu)$, where $p^\mu_\nu=\partial_\nu p^\mu$), so that the paring \eqref{inf-evaluation-d} reads
\begin{equation}\label{eq_PrimaCorrCOORD}
(v(x_0),v_\mu(x_0)), (p^\mu(x_0),p^\mu_\nu(x_0)) \longmapsto (v(x_0) p^\mu_\mu(x_0)+p^\mu (x_0)v_\mu (x_0) )\dd_{x_0}^n\x\, .
\end{equation}
Above coordinate expression \eqref{eq_PrimaCorrCOORD} shows that the coordinate--free formula \eqref{inf-evaluation-d} is well--defined, since its right--hand side depends only on the coordinates of $\zeta$ and $w$, and not on their respective extensions $j_1(p)$ and $j_1(v)$. However, \eqref{inf-evaluation-d} defines $\Psi$ only over the point $x_0$, so that the next step is to let  the point $x_0$ vary in \eqref{eq_PrimaCorrCOORD}, thus obtaining the expression
\begin{equation}\label{eqPsiRistretta}
\Psi_s: (p^\mu ,p^\mu_\nu)\longmapsto\left( p^\mu_\mu\delta \varphi +p^\mu \delta \varphi^\mu\right)\otimes \dd^n x
\end{equation}
of the \emph{restriction} $\Psi_s$  of $\Psi$ to the graph $j_1s$. Nevertheless, it is evident from \eqref{eqPsiRistretta} that
\begin{equation}
\Psi_s(\zeta)=\Psi_{s^\prime}(\zeta)
\end{equation}
whenever $s$ and $s^\prime$ have the same first--jet $q^I$ in $x_0$. So, the same formula \eqref{eqPsiRistretta} can be taken as the defining formula for the \emph{global} $\Psi$,
thus concluding the second part of the proof.\par
The last part is almost self--evident. Take an element $\eta\in V^\ast  \Q^I \otimes_{ \Q^I}   \Lambda^{n}M$ attached to the point $q^I$. Then,     $\Psi^{-1}(\eta)$ is the submanifold of $(J^1\P)_{q^I}$ described by the equations
\begin{eqnarray}
p^\mu &=&\textrm{const.}\quad \forall \mu\, ,\\
p^\mu _\mu&=&\textrm{const.}\,  ,
\end{eqnarray}
which, complemented with the equations of $(J^1\P)_{q^I}$, are precisely the same equations  which define    a degeneracy leaf of $\widetilde{\omega}^I$ (see Section \ref{ToyStep3}).
\end{proof}
Observe that the coordinate formula \eqref{eqPsiRistretta} of the canonical identification $\Psi$ is formally identical to the formula \eqref{thetaISenzaTilde} defining $\theta^I$.
\subsection{The constrained case}\label{sec52}
Theorem \ref{thEgr} can be easily adapted to the case when a constraint $\C$ is given in the infinitesimal configuration bundle $\Q^I$. More precisely, let  $\C\subseteq \Q^I$ be  a sub--bundle,   consider  its vertical bundle $V\C\longmapsto\C$, and  observe that
\begin{itemize}
\item the base $\C$ of $V\C$ is contained into the base $\Q$ of $V\Q$;
\item the generic fibre $T_{q^I}\C_{x^0}$ of $V\C$ is a linear subspace of the fibre $T_{q^I}\Q^I_{x^0}$ of $V\Q$.
\end{itemize}
These two facts are summarised by   the    commutative diagram \eqref{eqRestrizVertConstr}, where the lower inclusion corresponds to the restriction of the base and the upper one formalises the corresponding reduction of the fibres:
\begin{equation}\label{eqRestrizVertConstr}
\xymatrix{
V\C\ar@{^(->}[r] \ar[d] &  V\Q^I\ar[d]\\
\C\ar@{^(->}[r]   &    \Q^I\ .
}
\end{equation}
It is convenient to introduce the annihilator of $V\C$ in $\left.V^*\Q^I\right|_\C$, i.e., the subspace of the latter composed of  vertical covectors   vanishing on the former, which we shall denote by  $(V\C)^\circ  $.  For instance, if $\C$ is defined as in \eqref{constr}, then
\begin{equation}
(V\C)^\circ=\Span{\delta C_a\mid a=1,\ldots,k}\, ,
\end{equation}
i.e., as a module of sections, $(V\C)^\circ$ is generated by the $\delta C_a$'s, and we shall always assume that $\rank(V\C)^\circ=k$ or, in other words, that the constraints are independent.\par
 Let
 \begin{equation}\label{eqDefOmegaTildeCi}
\widetilde{\omega}^I_\C:=\left.\widetilde{\omega}^I\right|_{J^1{\P}|_\C}.
\end{equation}

\begin{corollary}\label{ConsEQthEgr}
 The restriction   $\Psi|_\C$  respects the fibrations over $\C$, i.e., it makes
\begin{equation}\label{eqDaDimostrareMODIF2}
\xymatrix{
J^1{\P}|_\C\ar@{->>}[rr]^{\Psi|_\C}\ar[dr]&&  V^\ast  \Q^I|_\C \otimes_{ \C}   \Lambda^{n}M\ar[dl]\\
&\C,&
}
\end{equation}
commutative, and the fibres of  $\Psi|_\C$ are the degeneracy leaves of   $\widetilde{\omega}^I_\C$.
\end{corollary}
\begin{proof}
 A straightforward consequence of Theorem \eqref{thEgr}.
\end{proof}
Observe that diagram \eqref{eqDaDimostrareMODIF2} alone is not sufficient to define the infinitesimal phase bundle $\P_\C^I$ of a constrained theory, i.e., the analogous of Theorem \ref{ThToyModelHI} for the constrained case does not follow immediately from Corollary \ref{ConsEQthEgr}. The reason is that the leftmost space in diagram \eqref{eqDaDimostrareMODIF2} is still too big, since it contains the gauge degrees of freedom, which need to be factored out. For this purpose we consider  the  following  diagram
\begin{equation}\label{eqShortExacSeq}
\xymatrix{
(V\C)^\circ   \ar@{^(->}[r]\ar[dr] &\left.V^*\Q^I\right|_\C  \ar@{->>}[r]\ar[d]& V^*\C\ar[dl]\\
& \C\ ,&
}
\end{equation}
dual   to \eqref{eqRestrizVertConstr} and superpose it with diagram \eqref{eqDaDimostrareMODIF2} above. As a result we obtain the following sequence of two projections,
\begin{equation}\label{eqCatena}
J^1{\P}|_\C \rightarrow  V^\ast  \Q^I|_\C \otimes_{ \C}   \Lambda^{n}M  \rightarrow V^\ast  \C \otimes_{ \C}   \Lambda^{n}M
\end{equation}
which we may call $\Psi_\C$, since it descended from $\Psi$. Now we are ready to generalise Theorem \eqref{thEgr} to the constrained case.
 \begin{corollary}\label{corEgr}
 The fibres of $\Psi_\C$ defined by \eqref{eqCatena} are precisely the degeneracy leaves of the form $\widetilde{\omega}_\C^I$ defined by \eqref{eqDefOmegaTildeCi}, so that the reduced infinitesimal space $\P_{\textrm{\normalfont reduced}}^I$ defined by \eqref{inf-red} is isomorphic to
 \begin{equation}
V^\ast  \C \otimes_{ \C}   \Lambda^{n}M
\end{equation}
and the induced form  $\omega_\C^I$ is symplectic.
\end{corollary}
\begin{proof}
 Straightforward.
\end{proof}

\section{The Poincar\'e--Cartan form for higher--order Lagrangian theories}\label{SecPoinCart}
The symplectic two--form \eqref{omega} on the \emph{infinitesimal} phase bundle ${\cal P}^I$ which we employed in Section \ref{SecToyModel} to write down the dynamics of a first--order Lagrangian theory is, strictly speaking, an     $(n+1)$--form, being, in fact, vector--density--valued. A very important structure, related to this one, is  an $(n+1)$--form defined on the {\em phase bundle} ${\cal P}$, firstly introduced by the founders of the Calculus of Variations (Caratheodory, Hermann Weyl, DeDonder, Dedecker) and later exploited in 1974 by one of us (JK) to define the so--called \emph{multi--symplectic} approach to canonical field theory \cite{multi, multi1}. Independently, Pedro Luis Garcia considered similar structures (see e.g.~\cite{MR0243768,MR0406246,MR0337224,MR0256681}). Later on, our multi--symplectic approach was used by many authors (see e.g.~\cite{GIM,GIM2,MR1694063,ForgerRomero05,Vey}).
This structure can be regarded as an analog of the last term of the so--called ``Tulczyjew triple'', namely $T^*T^*Q$ (see also \cite{2014arXiv1406.6503G} on this concern). It is entirely covered by the \emph{symplectic} structure discussed in this paper. To begin with, we first  clarify the basic properties of the multi--symplectic structure, whose negligence   has led many authors to critical errors.
\subsection{First--order Lagrangians: a reminder}
In classical, non--relativistic Mechanics, the Poicar\'e--Cartan form is a convenient tool to formulate the Hamiltonian description of the dynamics in a way which is Galilei--invariant. In fact, the infinitesimal symplectic form \eqref{eqDefOmegaI}, in the case of a   $1$--dimensional base $M$, reads
\begin{equation}\label{Hamilt-1}
    \omega^I = \left(  \delta  \dot{p} \wedge \delta q
+ \delta p \wedge \delta   \dot{q}
 \right) \otimes   \mbox{\rm d}t  \, ,
\end{equation}
where the space--time coordinates have been replaced by the unique time parameter $t$, the ``dot'' denoting the (unique) time derivative,  and the field variables $\varphi^K$ by the configuration variables $q^K$. Because  $\dim M= 1$, the momentum has only one component: $p_K = p_K^{\ \ 1}$. (As usual, we skip the index $K$ labelling the degrees of freedom of the system when it does not lead to any contradiction.) The corresponding canonical form \eqref{thetaISenzaTilde} reads
\begin{equation}\label{Hamilt-2}
    \theta^I = \left(  \dot{p}  \delta q
+  p  \delta   \dot{q}
 \right) \otimes   \mbox{\rm d}t  \,  .
\end{equation}
The Lagrangian density ${\cal L} = L {\rm d}t$ generates the dynamics ${\cal D} \subset {\cal P}$ according to the equation $\delta {\cal L} = \left. \theta^I \right|_{\cal D}$ (see Corollary \ref{corELeqsAreLagrSubsHI}) which, written down in terms of coordinates, reads
\begin{equation}\label{Hamilt-3}
    \delta L(q,\dot{q}) =  \dot{p}  \delta q
+  p  \delta   \dot{q} \, ,
\end{equation}
or, equivalently,
\begin{equation}\label{eqLagrFOMech}
\left\{\begin{array}{ccc}p & = & \frac{\partial L}{\partial \dot{q}}\, , \\\dot{p} & = & \frac{\partial L}{\partial q}\, .\end{array}\right.
\end{equation}
In terms of control--response relations, this means that in the $4N$--dimensional symplectic space ${\cal P}^I_x$, which is  parametrized by the  coordinates $(q,\dot{q},p ,\dot{p})$, we have chosen $(q,\dot{q})$ as control parameters. With this choice, the $2N$--dimensional Lagrangian submanifold ${\cal D}$ is described by \eqref{Hamilt-3}. The na\"ive Hamiltonian approach consists in replacing the velocities $\dot{q}$ by the momenta $p$ in the role of control parameters. Then the velocities become the response parameters and the corresponding description of dynamics follows {\em via} the   Legendre transformation
\begin{equation}\label{eqCompLagrTransFO}
p\delta \dot{q}=\delta (p\dot{q})-\dot{q}\delta p
\end{equation}
which, plugged into \eqref{Hamilt-3}, yields
\begin{equation}\label{Hamilt-4}
    - \delta \left( p\dot{q} - L(q,\dot{q}) \right) =  \dot{p}  \delta q -
    \dot{q} \delta  p \, .
\end{equation}
Observe that, in order to perform the Legendre transformation,   everything needs to be calculated on ${\cal D}$, i.e.,~the velocity has to be expressed in terms of the new control parameters $(p,q)$, and this can be accomplished by using the dynamics \eqref{eqLagrFOMech}. Substituting  $\dot{q} = \dot{q}(q,p)$ into the left hand side of \eqref{Hamilt-4} we finally obtain
\begin{equation}\label{Hamilt-5}
    - \delta H(p,q) =  \dot{p}  \delta q -
    \dot{q} \delta  p \, ,
\end{equation}
or, equivalently,
\begin{equation}\label{eqHamFOMech}
\left\{\begin{array}{ccc}\dot{q} & = & \frac{\partial H}{\partial {p}}\, , \\\dot{p} & = & -\frac{\partial H}{\partial q}\, .\end{array}\right.
\end{equation}
Such an approach, which is very convenient from a computational point of view, does not possess an intrinsic counterpart, since it highly depends upon the choice of the reference frame. Indeed, the ``velocity'' is not a geometric object in the bundle ${\cal Q}$. When we pass to another reference frame all the quantities used above  transform in an odd way. To illustrate this phenomenon consider, for instance, the transformation of the above structure under the Galilei transformation.
\begin{example}
Suppose that in one reference frame we have:
\begin{eqnarray*}
  L &=& L(t, q, \dot{q}) = \frac m2 {\dot q}^2 - U(t,q) \, , \\
  p &=& \frac{\partial L}{\partial {\dot q}} = m{\dot q} \, , \\
  H &=&  p{\dot q} - L = \frac m2 {\dot q}^2 + U(t,q) =
  \frac 1{2m} p^2 + U(t,q)   \,  ,
\end{eqnarray*}
where $U = U(t, q)$ is a potential. Now, let us perform the same construction in another reference frame, moving with velocity $V$ with respect to the previous frame. The new position variable equals
\[
        Q(t) = q(t) - V \cdot t \, ,
\]
and we have
\begin{eqnarray*}
  {\widetilde L} &=& {\widetilde L}(t, q, \dot{q}) = \frac m2 {\dot Q}^2 - U(t,Q + tV)
  = \frac m2 {\dot Q}^2 - {\widetilde U}(t,Q) \, , \\
  P &=& \frac{\partial {\widetilde L}}{\partial {\dot Q}} = m{\dot Q} =
  m{\dot q} - mV = p - mV \, , \\
  {\widetilde H} &=&  P{\dot Q} - {\widetilde L} = \frac 1{2m} P^2 + {\widetilde U}(t,q)  =
  \frac 1{2m} \left( p - mV \right)^2 + {\widetilde U}(t,q) = H - pV + \frac m2 V^2 \, .
\end{eqnarray*}
A cheap trick to collect both cases into a single, invariant structure consists in considering the  following (degenerate) $2$--form defined on the whole phase bundle:
\begin{equation}\label{eqToyMultiSimpaticaForma}
\Omega:=dp\wedge dq -dH\wedge dt \, .
\end{equation}
For some purposes one considers also its primitive, contact $1$--form $\Theta$ defined by
\begin{equation}\label{eqToyMultiSimpaticaForma-Th}
\Theta:=p dq -H dt \, .
\end{equation}
Observe now that, in the new reference frame, we have
\begin{eqnarray*}
  {\widetilde \Theta} &=& P \mbox{\rm d}Q - {\widetilde H} \mbox{\rm d}t
  = (p - mV) \mbox{\rm d} (q - tV)
   - \left( \frac 1{2m} \left( p - mV \right)^2 + {\widetilde U}
   \right) \mbox{\rm d}t
   \\
   &=& p \mbox{\rm d} q - H \mbox{\rm d} t - mV
   \mbox{\rm d} q + \frac m2 V^2 \mbox{\rm d} t = \Theta  - mV
   \mbox{\rm d} q + \frac m2 V^2 \mbox{\rm d} t \ne \Theta
   \, , \\
  {\widetilde \Omega} &=& \mbox{\rm d} {\widetilde \Theta} =
  \mbox{\rm d} \Theta = \Omega \,  .
\end{eqnarray*}
This means that, indeed, the form $\Omega$ is Galilei--invariant. Moreover, it carries the complete description of the dynamics. Namely, a section $\sigma$ of the bundle ${\cal P}$  is declared to be \emph{compatible} with the dynamics if it satisfies the   condition
\begin{equation}\label{eqExPoinCartOrdineUno}
\sigma^\ast(X {\textrm{\scalebox{1.5}{$\lrcorner$}}}  \Omega)=0\, ,\quad \forall X\in\frak{X}(\P)\, ,
\end{equation}
which is imposed on all the sections $\sigma$ of $\P$. It is easy to check that  \eqref{eqExPoinCartOrdineUno} is equivalent to the Hamilton equations \eqref{eqHamFOMech}. To this end, we use the coordinate description of the section $\sigma$, i.e.,  $M \ni t \longmapsto (q(t), p(t)) \in {\cal P}_t$. It is easy to see that for $X= \frac{\partial}{\partial p}$ equation \eqref{eqExPoinCartOrdineUno} is, indeed, equivalent to the first equation of \eqref{eqHamFOMech}, whereas the remaining equation is obtained for $X= \frac{\partial}{\partial q}$.
\end{example}
Now, we pass to a first--order field theory, i.e., we replace $t\leftrightarrow x^\mu$, $q \leftrightarrow\varphi$, and $p \leftrightarrow p^\mu$, and we take ${\cal L}=L{\rm d}^n x$, with $L=L(x^\mu,\varphi,\varphi_\mu)$. The Lagrange equations \eqref{eqLagrFOMech} are now replaced by  \eqref{defMom}--\eqref{eqEL} or, equivalently, by the unique equation  \eqref{eqQuasiEulLag}:
\begin{equation}\label{eqlagrangePrimoOrdineComplete}
   \delta L(\varphi,\varphi_\mu)  =
    j  \delta \varphi
   +  p^\mu  \delta   \varphi_{, \mu} \, .
\end{equation}
In terms of control--response relations, this means that in the symplectic space ${\cal P}^I_x$, which is  parametrized by the  coordinates $(\varphi,\varphi_\mu,p^\mu , j)$, we have chosen $(\varphi,\varphi_\mu)$ as control parameters, whereas $(p^\mu , j)$ are the response parameters. With this choice, the Lagrangian submanifold ${\cal D}$ is described by \eqref{eqlagrangePrimoOrdineComplete}. The na\"{\i}ve Hamiltonian approach consists in replacing the role of ``velocities'' $\varphi_{, \mu}$ and the momenta $p^\mu$ as control and response parameters. For this purpose we use the following formula:
\begin{equation}\label{eqCompLagrTransFOfields}
  p^\mu  \delta   \varphi_{, \mu} =
  \delta (p^\mu    \varphi_{, \mu}) - \varphi_{, \mu} \delta p^\mu \, ,
\end{equation}
in analogy  with  \eqref{eqCompLagrTransFO}. When plugged into \eqref{eqlagrangePrimoOrdineComplete}, it yields
\begin{equation}\label{Hamilt-field-1}
    - \delta \left( p^\mu    \varphi_{, \mu} - L(\varphi,\varphi_{, \mu}) \right) =
    j   \delta \varphi - \varphi_{, \mu} \delta p^\mu \, .
\end{equation}
To complete the Legendre transformation, everything needs to be calculated ``on shell'', i.e.,~on the dynamics submanifold ${\cal D}\subset {\cal P}^I$. This means that the velocities $\varphi_{, \mu}$ have to be expressed in terms of the new control parameters $(\varphi,p^\mu)$, with the help of equations \eqref{defMom}--\eqref{eqEL}. Substituting  $\varphi_{, \mu} = \varphi_{, \mu}(\varphi,p^\mu)$ into the left--hand side of \eqref{Hamilt-field-1} and denoting
\begin{equation}\label{eqDefAcca}
    H = H (\varphi,p^\mu) :=
    p^\mu    \varphi_{, \mu} - L (\varphi , \varphi_{, \mu} ) \, ,
\end{equation}
we finally obtain
\begin{equation}\label{Hamilt-field-2}
    - \delta H(\varphi,p^\mu) =
    j   \delta \varphi  - \varphi_{, \mu} \delta p^\mu \,  ,
\end{equation}
which is equivalent to the following system of PDEs:
\begin{eqnarray}
   \varphi_{, \mu} &=& \frac{\partial H}{\partial p^\mu} \, ,
   \label{H-1} \\
   j =  \partial_\mu p^\mu &=& - \frac{\partial H}{\partial \varphi}  \, .
  \label{H-2}
\end{eqnarray}
Many authors consider \eqref{H-1}--\eqref{H-2} the field--theoretic analogues  of the Hamilton equations \eqref{eqHamFOMech} and call the generating function $H$ the ``field Hamiltonian''. We stress, however, that it has {\em nothing to do} with what the physicists call the Hamiltonian---a quantity which measures the amount of energy carried by the field configuration and which is the generating function of the dynamics with respect to a completely different control mode!\par

Even if computationally appealing, the above construction depends heavily upon its non--geometric ingredients. In particular, splitting the jet $(\varphi ,\varphi_{, \mu})$ into the field $\varphi$ and the derivatives $\varphi_{, \mu}$ is completely artificial because the latter do not constitute any geometric object. The field derivatives are defined with respect to a trivialization of the bundle ${\cal Q}$ (i.e.,~with respect to a choice of coordinates $\varphi^K$ on it) and there is no simple transformation law which would describe how the  \emph{formulae} \eqref{H-1}--\eqref{H-2} transform under a change of the trivialization (parametrization). This corresponds to the non--invariance of the Hamiltonian particle dynamics with respect to the Galileian transformations.\par

There is, nevertheless, a way to construct an invariant, geometric object which corresponds to the above construction: the Poincar\'e--Cartan form or   the ``multi--symplectic'' $(n+1)$--form, analogous with \eqref{eqToyMultiSimpaticaForma}:
\begin{equation}\label{Eul-Poin}
    \Omega:= \mbox{\rm d} p^\mu \wedge \mbox{\rm d}  x^1 \wedge \cdots
    \underset{{\mu^{\textrm{th}}{\textrm{ place}}}}{\wedge \underbrace{ \mbox{\rm d} \varphi  } \wedge}\cdots \wedge \mbox{\rm d}  x^n  - \mbox{\rm d} H \wedge
    \mbox{\rm d}  x^1 \wedge \cdots
    \wedge \mbox{\rm d}  x^n\, .
\end{equation}
\begin{theorem}\label{thMultiSimpatico}
The form \eqref{Eul-Poin}
is defined unambiguously  on the phase bundle ${\cal P}$, i.e., it does not depend upon the choice of its trivialization and the choice of coordinates. Moreover, field equations \eqref{H-1}--\eqref{H-2} are equivalent to the condition \eqref{eqExPoinCartOrdineUno} imposed on sections $\sigma$ of ${\cal P}$.
\end{theorem}
\begin{proof}
We stress that the particular ingredients of   formula \eqref{Eul-Poin}, i.e.,  $p^\mu$, ${\rm d} \varphi$ and $H$, {\em do} depend upon trivialisation, like in Mechanics. Nevertheless,  $\Omega$ does \emph{not}---it can be easily checked  by a direct inspection (it was proved in \cite{KT79}). This is an analog of the fact that in Mechanics both the Hamiltonian $H=H(t,p)$ and the form $\mbox{\rm d} q$ depend upon the choice of the reference frame (trivialization of the space--time, treated as a bundle over the time axis) but, miraculously, the 2--form \eqref{eqToyMultiSimpaticaForma} is unambiguously defined.\par
Concerning field equations, it is sufficient to take a vertical $ X\in\frak{X}(\P)$, since  \eqref{eqExPoinCartOrdineUno} is identically satisfied by \emph{all} sections  $\sigma$, if  the  $X$'s appearing in it are assumed to be   tangent to the graphs of such sections. Choosing
\[
    X = \frac{\partial}{\partial p^\mu} \, ,
\]
we see that  \eqref{eqExPoinCartOrdineUno} implies \eqref{H-1}, whereas for
\[
    X = \frac{\partial}{\partial \varphi} \, ,
\]
we see that  \eqref{eqExPoinCartOrdineUno} implies \eqref{H-2}, which ends the proof.
\end{proof}

Na\"{\i}vely, one could think that the first part of \eqref{Eul-Poin}, namely the $(n+1)$--form
\begin{equation}\label{poli}
     \wp := \mbox{\rm d} p^\mu \wedge \mbox{\rm d}  x^1 \wedge \cdots
    \underset{{\mu^{\textrm{th}}{\textrm{ place}}}}{\wedge \underbrace{ \mbox{\rm d} \varphi  } \wedge}\cdots \wedge \mbox{\rm d}  x^n  \, ,
\end{equation}
is nothing but a better version of the canonical two--form \eqref{omega}, i.e.,
$ \omega = \left( \delta p^\mu \wedge \delta \varphi \right) \otimes \dens{\mu} $, where the covectors $\delta p^\mu$ and $\delta \varphi$, which are defined on $V{\cal P}$ only,  have been upgraded to the status  of regular covectors on ${\cal P}$ and, finally, the tensor product ``$\otimes$'' was replaced by the exterior product ``$\wedge$''. This analogy is the departing  point of the so called ``poli--symplectic'' approach, where objects like \eqref{poli} are used as the  basic building blocks of the theory. We stress, however, that the very notion of such a ``$2$--vertical'' form (terminology of DeDonder and Weyl) has no sense because $2$--vertical forms get mixed with $1$--vertical ones if we change the trivialisation of ${\cal Q}$. Only the specific combination \eqref{Eul-Poin} is invariant (see \cite{Ded53} for the detailed discussion).

\begin{remark}
  The field energy $E$ (or the corresponding density ${\cal E}:= E \cdot \mbox{\rm d}^n x$)
  is a generator of the dynamics of the time evolution of Cauchy data and has nothing to  do with the function $H$ above. To define the field energy, one needs to choose a foliation of the  space--time by $(n-1)$--dimensional leaves, parametrised by the time variable $t=x^n$. The corresponding Legendre transformation consists in replacing {\em not all the derivatives}, but only the time derivative $\varphi_{, n}$, by the corresponding momentum as a control parameter $p^n$ and retaining the   remaining space derivatives as control.
\end{remark}

\subsection{The Poincar\'e--Cartan form for theories with higher order Lagrangians}
For higher--order field theories we take  $\Q$  as in \eqref{eqDefQ}, and use the coordinates \eqref{eqDefQ_coord}. Recall that, in this case, the infinitesimal configuration bundle $\Q^I$ is \emph{not} $J^1\Q$, but rather its submanifold (see \eqref{eqIncluItJetSpaces}). So, given a $k\Th$ order Lagrangian $\L= L \cdot {\rm d}^nx$, with
\begin{equation}
L=L (x^\mu,\phi, \underset{1\leq |\overline{\mu}|\leq k }{\underbrace{\ldots,\phi_{\overline{\mu}},\ldots}})\, ,
\end{equation}
we regard $L$ as a function on $J^1\Q$, i.e.,
%
\begin{equation}
L=L (x^\mu,\phi, \underset{1\leq |\overline{\nu}|\leq k-1,|\mmu|=k-1 }{\underbrace{\ldots,\phi_{\overline{\nu}},\ldots,\phi_{\overline{\mu},\lambda}}})\, .
\end{equation}
Consider the function
\begin{equation}\label{sm-h}
    h = \sum_{ |\overline{\mu}|= k}p^{\mmu} \varphi_{\mmu}- {L} \ ,
\end{equation}
which is the Legendre transformation of $L$ with respect to the highest order momenta:
\begin{equation}\label{sm-h-1}
    h = h(x^\mu,\phi, \underset{1\leq |\overline{\nu}|\leq k-1,|\mmu|= k }{\underbrace{\ldots,\phi_{\overline{\nu}},\ldots,p^{\mmu}}}) \, .
\end{equation}
This means that the field equation \eqref{pierwsze} has been solved with respect to the highest order derivatives $\varphi_{\mmu}$, $|\mmu|= k $, and their value in definition \eqref{sm-h}  has been expressed in terms of the remaining variables:
\[
    \varphi_{\mmu}= \varphi_{\mmu}(x^\mu,\phi, \underset{1\leq |\overline{\nu}|\leq k-1,|\mmu|= k }{\underbrace{\ldots,\phi_{\overline{\nu}},\ldots,p^{\mmu}}}) \, .
\]
Observe that, by  replacing the highest--order derivatives by the highest--order momenta in the formula \eqref{theta^I}, we get  a new one--form
\begin{eqnarray}\label{theta^H}
  {\theta}^H
   &=& \left( j \delta \varphi
  + j^\mu   \delta \varphi_\mu +
  j^{ {\mu_1}{\mu_2}} \delta \varphi_{{\mu_1}{\mu_2}}
  + \cdots - \varphi_{{\mu_1} \dots {\mu_{k}} }
   \delta  j^{{\mu_1} \dots {\mu_{k}}} \right)
  \otimes
     \mbox{\rm d}^n x
\end{eqnarray}
such that
\begin{equation}\label{Th-H}
\theta^I= \theta^H + \delta \sum_{ |\overline{\mu}|= k}p^{\mmu} \varphi_{\mmu} \otimes
     \mbox{\rm d}^n x\, .
\end{equation}
Hence, field equations  \eqref{defDynamicsHI} can be rewritten in the following way:
\begin{equation}\label{defDynamicsHI-h}
\left. \theta^H \right|_{\cal D} = \left. \theta^I  \right|_{\cal D} - \delta \sum_{ |\overline{\mu}|= k}p^{\mmu} \varphi_{\mmu}\otimes
     \mbox{\rm d}^n x = \delta {\cal L} - \delta \sum_{ |\overline{\mu}|= k}p^{\mmu} \varphi_{\mmu}\otimes
     \mbox{\rm d}^n x =
- \delta h \otimes
     \mbox{\rm d}^n x
\end{equation}
or, equivalently,
\begin{eqnarray}
  \varphi_{{\mu_1} \dots {\mu_{k}}}
  &=& \frac{\partial h}{\partial  p^{({\mu_1} \dots {\mu_{k}})}}
  - 0\, ,\label{pierwsze-h}\\
  p^{({\mu_1}\dots{\mu_{k-1}})}
  &=& - \frac{\partial h}{\partial \varphi_{{\mu_1}\dots {\mu_{k-1}}}} -
  \partial_\lambda p^{{\mu_1}\dots {\mu_{k-1}} \lambda } \, , \label{nastepne-h} \\
  \dots &=& \dots \, , \nonumber\\
  p^{ ({\mu_1}{\mu_2})}
  &=& -\frac{\partial h}{\partial \varphi_{{\mu_1}{\mu_2}}}
  - \partial_\lambda p^{{\mu_1}{\mu_2}\lambda } \, ,\nonumber\\
  p^\mu
   &=& -\frac{\partial h}{\partial \varphi_{\mu}}
   -\partial_\lambda p^{\mu \lambda }\, ,\label{ostatnje-h}
    \\
    0 &=& -\frac{\partial h}{\partial \varphi}
    - \partial_\lambda p^{\lambda}\,  .\label{eqELHI-h}
\end{eqnarray}

Now, in analogy with \eqref{eqDefAcca},  define the ``Hamiltonian'':
\begin{equation}
 {H}:=\sum_{0\leq |\overline{\mu}|\leq k-1} p^{\mmu\lambda}    \varphi_{\mmu,\lambda}- {L}\, ,
\end{equation}
or, equivalently,
\begin{equation}
 {H}:= \sum_{0\leq |\overline{\nu}|\leq k-1} p^{\nnu}\phi_{\nnu}+ \left( \sum_{ |\overline{\mu}|= k}p^{\mmu} \varphi_{\mmu}- {L} \right)
 = \sum_{0\leq |\overline{\nu}|\leq k-1} p^{\nnu}\phi_{\nnu}+ h \, .
\end{equation}
The function $H$ depends on $J^{k-1}\Phi$ and all the momenta, i.e.,~is defined on the phase bundle in both versions of the theory, i.e.,~${\cal P}$ and ${\cal S}$.
\begin{equation}
H=H (x^\mu,\phi, \underset{1\leq |\overline{\nu}|\leq k-1,0\leq|\mmu|\leq k }{\underbrace{\ldots,\phi_{\overline{\nu}},\ldots,p^{\mmu}}})\, .
\end{equation}

In analogy with \eqref{Eul-Poin}, we define the ``multi--symplectic'' $(n+1)$--form
\begin{equation}\label{Eul_Poin_HO}
   \widetilde{ \Omega}:= \sum_{0\leq|\mmu|\leq k-1}\mbox{\rm d} p^{\mmu\lambda} \wedge \mbox{\rm d}  x^1 \wedge \cdots
    \underset{{\lambda^{\textrm{th}}{\textrm{ place}}}}{\wedge \underbrace{ \mbox{\rm d} \varphi_{\mmu}  } \wedge}\cdots \wedge \mbox{\rm d}  x^n  - \mbox{\rm d} \widetilde{H} \wedge
    \mbox{\rm d}  x^1 \wedge \cdots
    \wedge \mbox{\rm d}  x^n\ .
\end{equation}
\begin{theorem}\label{thMultiSimpatico_HO}
The form \eqref{Eul_Poin_HO}
is defined unambiguously  on the phase bundle ${\cal P}$, i.e.,~it does not depend upon the choice of trivialisation and coordinates. Moreover, a section $\sigma$ of ${\cal P}$ satisfies the ``multi--symplectic equation'' \eqref{eqExPoinCartOrdineUno} if and only if it is:
\begin{enumerate}
\item holonomic,
\item   satisfies the above  field equations \eqref{pierwsze-h}--\eqref{eqELHI-h}.
\end{enumerate}
\end{theorem}

\begin{proof}
We stress again that the particular ingredients of the formula \eqref{Eul_Poin_HO} {\em do} dependent upon trivialization, like in Mechanics or in the first--order theory, but  $\Omega$ does not.
\par
To prove equivalence of the ``multi--symplectic equation'' \eqref{eqExPoinCartOrdineUno} with the field equations, we are allowed again to consider only  \emph{vertical} vectors. Begin with $X=\frac{\partial}{\partial \phi}$. Then, the quantity $\sigma^\ast(X  {\textrm{\scalebox{1.5}{$\lrcorner$}}}  \Omega)$ becomes


 \begin{equation}
0= - \partial_\lambda p^{\lambda} - \frac{\partial  {H}}{\partial \phi}=- \partial_\lambda p^{\lambda} - \frac{\partial  {h}}{\partial \phi}  \, ,
\end{equation}
i.e.,~equation  \eqref{eqELHI-h} is reproduced. For $X=\frac{\partial}{\partial \phi_{\mu}}$, where $1\leq|\mmu|\leq k-1$, we obtain
 \begin{equation}
0= - \partial_\lambda p^{\mmu\lambda} - \frac{\partial  {H}}{\partial \phi_{\mmu}}=- \partial_\lambda p^{\mmu\lambda} - \frac{\partial  {h}}{\partial \phi_{\mmu}} - p^{\mmu} \,  ,
\end{equation}
i.e.,~equations \eqref{nastepne-h} -- \eqref{ostatnje-h} are recovered.\par

Pass now to the derivatives with respect to momenta. For $X=\frac{\partial}{\partial p^{\mmu\lambda}}$, with $|{\mmu\lambda} | \leq k-1$, we obtain
\begin{equation}
0= \partial_\lambda \phi_{\mmu}-\phi_{\mmu\lambda} \, ,
\end{equation}
i.e.,~the holonomy constraints \eqref{pocz}--\eqref{kon} are reconstructed.
Finally, for the highest order momenta: $X=\frac{\partial}{\partial p^{\mmu}}$, with $|\mmu|=k$, we obtain
\begin{equation}
0=\phi_{\mmu}-\frac{\partial H}{\partial p^{\mmu}}=\phi_{\mmu}-\frac{\partial h}{\partial p^{\mmu}} \, ,
\end{equation}
i.e.,~equation  \eqref{pierwsze-h} is recovered.
\end{proof}

\section{Appendices}

The claim of   Proposition \ref{ProposizioneFondamentale}, which was fundamental in our analysis,   was left without proof, since  in local coordinates it reduces to a trivial shuffling of variables (see formula \eqref{eqJ1VugualeVJ1COORD}). Nevertheless, even if coordinates were thoroughly  exploited to provide workable formulae, the framework we presented here has an invariant character, and  it cannot be concluded without   an global proof of the key identification \eqref{eqJ1VugualeVJ1}.

\subsection{A proof of the identification $VJ^1=J^1V$}\label{appVJ1UgJ1V}

Such an equivalence can be framed in the general context of (nonlinear) differential operator between fibreed manifolds, nonlinear PDEs, and their (infinitesimal) symmetries (see \cite{MR2813504} and references therein).
\subsubsection{Geometry of (nonlinear) differential operators}
Let $\pi:\Q\to M$ and $\eta: \P\to M$ be two smooth bundles over the same base manifold $M$. Denote by $J^1\Q$ the first jet prolongation of $\pi$. A smooth fibreed mapping
\begin{equation}
\xymatrix{
J^1\Q\ar[rr]^\kappa\ar[dr]_{\pi_1} & & \P\ar[dl]^\eta\\
& M
}
\end{equation}
is called a \emph{(nonlinear) $1\St$ order differential operator between $\pi$ and $\eta$}.\par
The reason of this definition is obvious. A (local) section $s\in\Gamma(\pi)$ determines a (local) section $j_1(s)$ of $J^1\Q$; since $\kappa$ maps (local) sections into (local) sections, the composition $\kappa\circ j_1(s)$ is a (local) section of $\eta$. In other words, $\kappa$ determines the map
\begin{eqnarray*}
\Gamma(\pi) &\stackrel{\Delta_\kappa}{\longrightarrow} & \Gamma(\eta),\\
s &\longmapsto & \kappa\circ j_1(s).
\end{eqnarray*}
Let now $\phi$ be a fibre coordinate on $\Q$ (which for simplicity is supposed of rank one); then a section $s$ corresponds to a function $\phi=\phi(x^1,\ldots,x^n)$, and its $1\St$ jet $j_1(s)$ to an $(n+1)$--tuple of functions $\phi,\phi_\mu$, where $\phi_\mu:=\frac{\partial\phi}{\partial x^\mu}$.\par
On the other hand (assuming $\eta$ to be of rank one as well), $\Delta_\kappa(s)$ is determined by the unique function
\begin{equation}
\kappa\left(x^1,\ldots,x^n,\phi(x^1,\ldots,x^n),\frac{\partial\phi}{\partial x^1}(x^1,\ldots,x^n),\ldots,\frac{\partial\phi}{\partial x^n}(x^1,\ldots,x^n)\right),
\end{equation}
which resembles the familiar expression of a (nonlinear) $1\St$ order differential operator. The ``{inverse}'' holds as well, in the following sense.
\begin{proposition}\label{propRappresentabilita}
Let $\widetilde{U}\subseteq \Q$ be an open sub--bundle over an open subset  $U\subseteq M$, with abstract fibre $\R^m$, where $m=\rank \Q$, and $\widetilde{V}\subseteq \P$ an open sub--bundle over the same $U$, with abstract fibre $\R^l$, where $l=\rank  P$. Regard elements of $C^\infty(U,\R^m)$   as local sections of $\pi$. Let $s\in C^\infty(U,\R^m)$ such that its image $\Delta(s)$ sits in $C^\infty(U,\R^l)$; hence, $\Delta(s)$ identifies with an $l$--tuple of functions $(\kappa^1\ldots,\kappa^l)$ on $U$, and $s$ identifies with an $m$--tuple of functions $(\phi^1,\ldots,\phi^m)$ on $U$. Suppose that, for all $\widetilde{U}, \widetilde{V}$, and $s$, it holds
\begin{equation}
\kappa^j=\kappa^j\left(x^1,\ldots,x^n,\phi^k(x^1,\ldots,x^n),\ldots,\frac{\partial\phi^k}{\partial x^i}(x^1,\ldots,x^n),\ldots \right).
\end{equation}
Then a unique smooth fibreed mapping $\kappa:J^1\Q\to \P$ exists, such that $\Delta=\Delta_\kappa$.
%
%
%
%
%
%
%
%
 \end{proposition}
Proposition \ref{propRappresentabilita} suggests that, in order  to construct a smooth fibreed mapping from $J^1\Q$ to $\P$, one may equivalently look for a $1\St$ order differential operator from $\Gamma(\pi)$ to $\Gamma(\eta)$.

\subsubsection{Lifting of vertical symmetries} Let now $v:V\Q\to \Q$ be the vertical tangent bundle, and $X$ an its section. In coordinates, $X=\psi\frac{\partial}{\partial \phi}$. Look for a vertical vector field $\widetilde{X}$ on $J^1\Q$, i.e., a section of $v_1:VJ^1\Q\to J^1\Q$, such that
\begin{itemize}
\item $\widetilde{X}$ is a lifting of $X$,
\item $\widetilde{X}$ is an infinitesimal contact transformation of $J^1\Q$.
\end{itemize}
Then, it can be easily proved that
\begin{equation}\label{eqFormulaSoll}
\widetilde{X}=\psi\frac{\partial}{\partial \phi} + D_\mu(\psi) \frac{\partial}{\partial \phi_\mu},
\end{equation}
where $D_\mu=\frac{\partial}{\partial x^\mu}+\phi_\mu \frac{\partial}{\partial \phi}$ is the (truncated) total derivative operator: the first condition dictates the coefficient of $\frac{\partial}{\partial \phi}$, while the second that of $ \frac{\partial}{\partial \phi_\mu}$. Let us give a geometric interpretation to the lifting procedure
\begin{equation}
X\longmapsto\widetilde{X}\label{eqLiftProc}
\end{equation}
To this end, consider the 1--parameter group of transformations $\psi_t$ determined by $X$. Being $X$ vertical, each $\psi_t$ is a fibreed morphism over the identity $\mathrm{id}_M$, i.e., $s_t:=\psi_t\circ s$ is a family of sections of $\pi$, such that $s_0=s$ (somebody calls it a \emph{vertical homotopy} \cite{MR3103143}). All these sections can be, so to speak, ``{jettified}'', thus obtaining a family of sections $j_1(s_t)$ of $J^1\Q\to M$. Fix a point $\x=(x^1,\ldots,x^n)\in M$, and compute the velocity of the curve $t\longmapsto j_1(s_t)(\x)$ at zero:
\begin{equation*}
v_{\x}:=\left.  \frac{\mathrm{d} j_1(s_t)(\x)}{\mathrm{d} t} \right|_{t=0}.
\end{equation*}
By construction, $v_{\x}$ is a vertical tangent vector on $J^1\Q$, at the point $j_1(s)(\x)$, and direct computations show that $$\widetilde{X}_{j_1(s)(\x)}=v_{\x}.$$

\subsubsection{Proof of the equivalence}

We use now   \eqref{eqLiftProc} to define a $1\St$ order differential operator between $V\Q\longrightarrow M$ and $VJ^1\Q\longrightarrow M$.

\begin{equation}
\xymatrix{
V\Q\ar[r]^v\ar[d]  & \Q\ar[dl]^\pi\ar@{.>}@/_1pc/[l]_{\check{X}}\\
M\ar@{.>}@/_1pc/[ur]_{s_X}\ar@{.>}@/^1pc/[u]^{X}&
}\label{diagrammaunpoaffollato}
\end{equation}
Observe that the bundle $V\Q\longrightarrow M$ is the composition $\pi\circ v$: according, a section of  $V\Q\longrightarrow M$ is, in a sense, the composition
%
%
$X=\check{X}\circ s_X$ of two sections, where $s_X:=v\circ X$ is a uniquely defined section of $\pi$ and  $\check{X}$ is a section of $v$ which is \emph{not} unambiguously defined: just its restriction  $\check{X}|_{\textrm{Im}\, s_X}$ is uniquely determined by $X$ (see diagram \eqref{diagrammaunpoaffollato}). So, one may work with the pair $(s_X,\check{X})$, instead of $X$, bearing in mind the ambiguity of $\check{X}$.\par
For instance, in local coordinates,  $\check{X}=\psi\frac{\partial}{\partial \phi}$, where only the restriction $\psi|_{\textrm{Im}\, s_X}$ is uniquely determined by $X$. Now, bearing in mind formula \eqref{eqFormulaSoll}, $\check{X}$ can be lifted to
\begin{equation}
\widetilde{\check{X}}:=\psi\frac{\partial}{\partial \phi}+D_\mu(\psi) \frac{\partial}{\partial \phi_\mu},
\end{equation}
and $s_X$ can be prolonged to a section $j_1(s_X)$ of $VJ^1\Q\longrightarrow J^1\Q$. Hence, we can produce a section $\widetilde{X}$ of $VJ^1\Q\longrightarrow M$ by putting $\widetilde{X}:=\widetilde{\check{X}}\circ j_1(s_X)$:
\begin{equation}\label{diagrammaunpoaffollato2}
\xymatrix{
VJ^1\Q\ar[r]\ar[d]  & J^1\Q\ar[dl]\ar@{.>}@/_1pc/[l]_{\widetilde{\check{X}}}\\
M\ar@{.>}@/_1pc/[ur]_{j_1(s_X)}\ar@{.>}@/^1pc/[u]^{\widetilde{X}}&
}
\end{equation}
Diagram \eqref{diagrammaunpoaffollato2} illustrates the relationship between sections $\widetilde{\check{X}}$, $j_1(s_X)$  and $\widetilde{X}$. The latter can be computed directly,
\begin{equation*}
\widetilde{X}_{\x}=\widetilde{\check{X}}_{j_1(s_X)(\x)}=\left.\psi(j_1(s_X)(\x))\frac{\partial}{\partial \phi}\right|_{j_1(s_X)(\x)}+\left.\frac{\partial(\psi\circ j_1(s_X))}{\partial x^\mu}(\x) \frac{\partial}{\partial \phi_\mu}\right|_{j_1(s_X)(\x)}
\end{equation*}
thus showing that
$\widetilde{X}$ depends only on $s_X$ and $\check{X}|_{\textrm{im}\, s_X}$, i.e., that the map $X\longmapsto\widetilde{X}$ is well--defined.\par
Observe that the fibre $V_\x \Q$ is the tangent manifold $ T\Q_\x$ of the fibre $\Q_\x$: hence, if $\phi$ is a coordinate on the abstract fibre of $\pi$, and $p$ its conjugate momentum, the section  $X$ is given, in local coordinates, by
\begin{equation*}
M\ni\x\stackrel{X}{\longmapsto} (\phi(\x),p(\x))\in V\Q.
\end{equation*}
According,
\begin{equation*}
M\ni\x\stackrel{s_X}{\longmapsto} (\x,\phi(\x) )\in Q
\end{equation*}
and $\psi\circ j_1(s_X)=p$.
Similarly, the fibre  coordinates of $VJ^1\Q\longrightarrow M$ are $(\phi,\phi_\mu,p,p_\mu)$. The section $\widetilde{X}$, in such coordinates, is given by
\begin{eqnarray}
\phi_\mu &:=& \frac{\partial\phi}{\partial x^\mu},\label{eqIcsTilde1}\\
p_\mu &:=& \frac{\partial p}{\partial x^\mu}.\label{eqIcsTilde2}
\end{eqnarray}
\begin{corollary}[Proof of the identification]
The lifting procedure \eqref{eqLiftProc} is a $1\St$ order differential operator, and the corresponding smooth fibreed mapping $\kappa: J^1V\Q\longrightarrow VJ^1\Q$ is one--to--one.
\end{corollary}
\begin{proof}
Formulae  \eqref{eqIcsTilde1}--\eqref{eqIcsTilde2} shows that   \eqref{eqLiftProc} is a $1\St$ order differential operator. Then   Proposition \ref{propRappresentabilita} allows to associate with it   the smooth fibreed mapping $\kappa$, and it remains to prove that $\kappa$ is one--to--one.\par
This can be accomplished locally, by observing that the fibre coordinates on $J^1V\Q$ are $(\phi,p,\phi_\mu,p_\mu)$, so  that the map $\kappa$ simply ``{flips}'' $\phi_\mu$ and $p$. In particular, $\kappa$ is (locally) one--to--one.
\end{proof}
 This \virg{flipping}, which occurs due to the interchanging of the \virg{jettification} and the \virg{verticalization} procedures,  namely,
 \begin{center}
 {\footnotesize
\begin{tabular}{c|ccccc}
$ VJ^1\Q$ &  $\underset{\textrm{fibre variable}}{\phi}$ & $\Rightarrow$  & $\underset{\textrm{jettified fibre variable}}{\phi_\mu}$ & $\Rightarrow$ &  $\underset{\textrm{vertical momenta}}{(p,p_\mu)}$ \\
\hline
 $J^1V\Q$ &  $\underset{\textrm{fibre variable}}{\phi}$ &$\Rightarrow$ & $\underset{\textrm{vertical momentum}}{p}$ & $ \Rightarrow $ & $ \underset{\textrm{jettified fibre variable \& its momentum}}{(\phi_\mu,p_\mu)}$
\end{tabular}
}
 \end{center}
can be regarded as a jet--theoretic analog of the last term of the \virg{Tulczyjew triple}.

%
%
%
%
%
%
%
%
%

\subsection{The space of symmetric momenta}\label{app-momentum-gauge}
Here we show that the sub--bundle $\mathcal{S}$ mentioned in \ref{momentum-gauge} can be characterized intrinsically.\par Roughly speaking the symmetrisation of the momenta $p^{\mu_1\cdots\mu_{k-1}\lambda}$ with respect to the \emph{last} index $\lambda$ is a manifestation of the so--called \emph{polarization} of homogeneous polynomials. Classically, it is used, among many other things, to compute the tangent space to a quadric surface, but it keeps finding unexpected applications, especially in the framework of jet spaces (see, e.g., \cite{MorBach}). The reason is that the spaces of homogeneous polynomials   are the linear models of   the jet bundles (which are affine)
 and their polarization correspond to the immersion into \emph{nonholonomic jets}, i.e., precisely those used in our approach to higher--order theories (see, e.g., \cite{MR861121,MR989588}).

\subsection{Polarization and Spencer operator}

The simplest example of a \emph{polarization} is that of a quadratic form $Q(x)$, i.e., its    corresponding bilinear form $B$  defined by
\begin{equation}\label{eqQuasiInutile}
B(x,y):=\frac{Q(x+y)-Q(x)-Q(y)}{2}.
\end{equation}
Intrinsically, formula \eqref{eqQuasiInutile} reads
\begin{equation}\label{eqQuasiInutile2}
B=\frac{1}{2}dQ.
\end{equation}
Indeed, if $Q=\alpha x^2+\beta xy+\gamma y^2$ then its differential    $dQ=(2\alpha x+\beta y)dx+(\beta x+2\gamma y)dy $ correspond to (twice) the symmetric $2\times 2$ matrix  of the form $B$, i.e., to
\begin{equation*}
 \left(\begin{array}{cc}2\alpha & \beta \\\beta & 2\gamma\end{array}\right)\, .
\end{equation*}
The advantage of \eqref{eqQuasiInutile2} against \eqref{eqQuasiInutile} is that the former admits a straightforward generalization to cubic forms, quartic forms, etc. Indeed, for any homogeneous   polynomial   $p$ of degree $k$ in the $n$ independent variables $x_1,\ldots,x_n$, the differential  $dp$ is a linear combination of the $x_i$'s with values in the space of polynomials of  degree $(k-1)$.\footnote{Yet another way to understand this is through the so--called \emph{meta--symplectic form} (see the above cited  \cite{MorBach} on this concern).} If $V$ denotes the linear space generated by the $x_i$'s, then the operation $p\longmapsto \frac{1}{k}dp$ is nothing but the canonical inclusion
\begin{equation}\label{eqQuasiInutile3}
S^kV \subseteq S^{k-1}V \otimes_\R V \, ,
\end{equation}
where ``$S$'' stand for ``symmetric power''. Written  down   in coordinates, \eqref{eqQuasiInutile3} is the passage from a space where \emph{all} indices are symmetric to one where so are only the first $k-1$.\par In the jet--theoretic context, \eqref{eqQuasiInutile3} is the \virg{linearization} (i.e., the tangent mapping) of the embedding of the holonomic $k\Th$ jets into the nonholonomic  ones (i.e., the $1\St$ jets of $k-1\St$ jets), i.e., formula \eqref{jets-in-jets}.

\subsection{Intrinsic definition of $\cal S$}
As a preliminary observation, recall that the cotangent manifold $TW$ of a  linear space $W$ is the trivial bundle $T^*W=W\times W^*$. The key property of $Q_x$ we shall need here is that $Q_x$ is an affine bundle over $Q'_x:=(J^{k-2}\Phi)_x$, whose linear model is
\begin{equation}
W_x:=S^{k-1}T_x^*M\otimes_\R\Phi_x^* \ .
\end{equation}
%
The cotangential mapping of the canonical  projection  $Q_x\longrightarrow Q'_x$, allows to project $T^*Q_x$ over $T^*Q'_x$ as well. The generic fibre is now
\begin{equation}\label{fibra_fibrosa}
T^*(S^{k-1}T_x^*M\otimes\Phi_x^*)=W_x\times S^{k-1}T_xM\otimes_\R\Phi_x.
\end{equation}
On the other hand, by Poincar\'e duality,
\begin{equation}\label{poinc}
\stackrel{n-1}{\bigwedge}T_x^*M\equiv T_xM\otimes_\R \stackrel{n}{\bigwedge}T_x^*M.
\end{equation}
Together, \eqref{fibra_fibrosa} and \eqref{poinc} imply that  $T^*Q_x\otimes_\R \stackrel{n-1}{\wedge}T_x^*M$  projects  onto  $T^*Q'_x\otimes_\R \stackrel{n-1}{\wedge}T_x^*M$, with generic fibre
\begin{equation}\label{eqCheDefinisceEsse}
W_x\times (S^{k-1}T_xM\otimes\Phi_x)\otimes_\R T_xM\otimes\Lambda^nT_x^*M
\end{equation}
which, in view of the polarization/Spencer operator \eqref{eqQuasiInutile3} contains a canonical subspace ${\cal S}_x$ obtained by replacing  $S^{k-1}T_xM\otimes T_xM$ with  $S^kT_xM$ in \eqref{eqCheDefinisceEsse}. By arbitrariness of $x\in M$, this defines a whole bundle $\cal S$.\par
It is worth stressing the unambiguity of such a definition: even if $W_x$ is \emph{not} canonically identified with the fibre of $Q_x\longrightarrow Q'_x$ (an origin is needed),   it was shown above that  \emph{for any}   choice of $W_x$, the corresponding fibre  \eqref{eqCheDefinisceEsse}  contains a  \emph{unique} distinguished  subspace.

\subsection{List of main symbols}\label{subListSymb}
\begin{center}
{\footnotesize
\begin{tabular}{r|l}

$M$ & the space--time\\
$x,x_0$ &  a generic (resp., fixed) point of $M$\\
$x^\mu$  & coordinates on $M$\\
$\Q,\Phi\stackrel{\pi}{\longrightarrow}M$ & the configuration bundle for first (resp., higher) order theories\\
${\cal C}\subseteq \Q^I$ & the constrained sub--bundle\\
$C_a$ & the constraint functions\\
$\lambda^a$ & Lagrange multipliers\\
$\varphi^K$ & fibre coordinates on $\Q$\\
$V\Q$ & the vertical tangent bundle\\
$\Lambda^iM$ & the bundle of differential $i$--forms on $M$\\
$\P$ & the phase bundle.\\
$\P^I$ & the infinitesimal  phase bundle\\
$\P^I_{\cal C}$ & the constrained infinitesimal  phase bundle\\
$\P^I_{\textrm{reduced}}$ & the reduced infinitesimal  phase bundle

 \end{tabular}
 }
\end{center}

\begin{center}
{\footnotesize
\begin{tabular}{r|l}

$\cal D$ & the dynamics\\
${\cal S}\subseteq P$ & the symmetric phase bundle\\
$p^\mu$ & the momenta\\
$r^\mu$ & extra momenta\\
$s^{{\mu_1} \dots {\mu_{l}}} $ & symmetric momenta\\
${\mathrm d}^n\x$ & the volume element on $M$\\
$\dens{\mu}$ & basis of vector--densities on $M$\\
$B\to M$ & a bundle over $M$\\
$B_x$ & fibre of $B$ at $x\in M$\\
$\frak{X}(B)$ & vector fields on $B$\\
$d$ & exterior derivative\\
$ {\mathrm d}$ & space--time differential\\
$\delta$ & vertical differential\\
$\partial_\mu$ & total derivative\\
$\theta$ & Liouville form\\
$\omega^I$ & pre--symplectic form\\
$\omega^I_{\cal C}$ & constrained pre--symplectic form\\
$\omega$ & symplectic form\\
$J^1\P$ & first jet--extension of $\P$\\
$\Q^I$ & the infinitesimal configuration bundle\\
$\Pi:\P\to \Q$ & canonical projection\\
${\cal L}$ & a Lagrangian density\\
$L$ & a Lagrangian function\\
$H$ & Hamiltonian function\\
$h$ & Legendre transform\\
 $j$ &   current\\ $\overline{\mu}$ & a multi--index\\
 $\varphi_{\overline{\mu}}$ & coordinates on $J^{k-1}\Q$\\
 $j^{\overline{\mu}}$ & higher--order current\\
 $\Omega$ & Poincar\'e--Cartan form\\
 $\Theta$ & primitive contact form\\

 $J^1\P\stackrel{\Psi}{\longrightarrow} V^\ast \Q^I\otimes_{\Q^I}\Lambda^n M$ & canonical surjective mapping\\
 $\zeta$ & a point of $J^1\P$\\
  $q^I$ & a point of $\Q^I$\\
 $p$ & a section of $\P$\\
 $s$ & a section of $\Q$\\
 $w$ & a vertical vector on $\Q$\\
 $v$ & a section of $V\Q$\\
 $<\, \cdot\, ,\, \cdot\, >$ & volume--forms--valued pairing of vectors and covectors\\
 $\eta$ &  a point of $V^\ast \Q^I\otimes_{\Q^I}\Lambda^n M$.

 \end{tabular}
 }
\end{center}

%

 \bibliographystyle{plainnat}
\bibliography{sympl-higher-order}

\end{document}